\pdfoutput=1

\documentclass{amsart}

\usepackage[mathscr]{eucal}
\usepackage{amssymb}
\usepackage[usenames,dvipsnames]{xcolor} 
\usepackage[normalem]{ulem}
\usepackage{amsthm}
\usepackage{bbold}
\usepackage{comment}
\usepackage{enumitem}
\usepackage{amsmath}
\usepackage{wasysym}
\usepackage{tikz-cd}
\usepackage{calc}

\usepackage{etoolbox} 

\usepackage{stmaryrd}

\usepackage[unicode]{hyperref} 

\definecolor{dark-red}{rgb}{0.5,0.15,0.15}
\definecolor{dark-blue}{rgb}{0.15,0.15,0.6}
\definecolor{dark-green}{rgb}{0.15,0.6,0.15}
\hypersetup{
    colorlinks, linkcolor=Blue,
    citecolor=Blue, urlcolor=Blue
}

\usepackage[nameinlink,capitalise,noabbrev]{cleveref}
\usepackage[hyphenbreaks]{breakurl}

\numberwithin{equation}{section}
\usepackage[all]{xy}
\xyoption{line}
\usepackage{graphicx}
\usepackage{mathtools}
\SelectTips{cm}{10}

\newtheorem{Thm}[subsection]{Theorem}
\newtheorem*{Thm*}{Theorem}
\newtheorem{Prop}[subsection]{Proposition}
\newtheorem{Lem}[subsection]{Lemma}
\newtheorem{Cor}[subsection]{Corollary}

\newtheorem{thmx}{Theorem}

\theoremstyle{remark}
\newtheorem{Def}[subsection]{Definition}
\newtheorem{Ter}[subsection]{Terminology}
\newtheorem{Rec}[subsection]{Recollection}
\newtheorem{Not}[subsection]{Notation}
\newtheorem{Exa}[subsection]{Example}

\newtheorem{Rem}[subsection]{Remark}
\newtheorem{Par}[subsection]{}
\newtheorem{Que}[subsection]{Question}

\let\realequation\equation
\def\equation{\setcounter{equation}{\arabic{subsection}}%
   \refstepcounter{subsection}%
   \realequation}

\usepackage{tikz}
\usepackage{tikz-cd}
\usetikzlibrary{trees}
\usetikzlibrary[shapes]
\usetikzlibrary[arrows]
\usetikzlibrary{patterns}
\usetikzlibrary{fadings}
\usetikzlibrary{backgrounds}
\usetikzlibrary{decorations.pathreplacing}
\usetikzlibrary{decorations.pathmorphing}
\usetikzlibrary{positioning}
\usetikzlibrary{shapes.geometric}

\tikzstyle{category} = [rectangle, rounded corners, minimum width=3cm, minimum height=1cm, text centered, text width=2.5cm, draw=black]
\tikzstyle{area} = [rectangle, minimum width=4cm, minimum height=1cm, text centered, text width=3.8cm, draw=black]

\tikzstyle{arrow} = [thick,->,>=stealth]
\tikzstyle{arrow2} = [thick,->,>=stealth,dotted]

\newcommand{\nc}{\newcommand}
\nc{\dmo}{\DeclareMathOperator}

\usepackage{todonotes}

\renewcommand{\emptyset}{\varnothing}
\nc{\eqname}[2][0pt]{\makebox[#1][r]{(#2)}}

\nc{\overbar}[1]{\mkern 1.5mu\overline{\mkern-1.5mu#1\mkern-1.5mu}\mkern 1.5mu}

\nc{\kappaaux}{g}
\nc{\kappam}{{\kappaaux({\frak m})}}
\nc{\kappaP}{{\kappaaux(\cat P)}}
\nc{\kappaQ}{{\kappaaux(\cat Q)}}
\nc{\kappaSP}{{\kappaaux_{\cat S}(\cat P)}}
\nc{\kappaTP}{{\kappaaux_{\cat T}(\cat P)}}
\nc{\kappaSQ}{{\kappaaux_{\cat S}(\cat Q)}}
\nc{\kappaTQ}{{\kappaaux_{\cat T}(\cat Q)}}
\nc{\kappaphiB}{{\kappaaux(\varphi(\cat B))}}
\nc{\kappaphiQ}{{\kappaaux(\varphi(\cat Q))}}

\dmo{\Sub}{Sub}
\nc{\SpEn}{\cat S_{E(n)}}
\nc{\SpEnf}{\cat S_n}
\dmo{\Loc}{Loc}
\dmo{\Locideal}{Locid}
\dmo{\Colocideal}{Colocid}
\dmo{\Thickideal}{Thickid}
\nc{\Locs}[1]{\Loc\langle #1 \rangle}
\nc{\Loco}[1]{\Locideal\langle #1 \rangle}
\nc{\Coloco}[1]{\Colocideal\langle #1 \rangle}
\nc{\bbullet}{{\scriptscriptstyle\hspace{-1pt}\bullet}}
\nc{\bullett}{{\scriptscriptstyle\bullet}\hspace{-1pt}}
\nc{\LF}{L\hspace{-0.2ex}F}
\nc{\SpG}{\Sp_G}
\nc{\SpGn}{\Sp_{G,n}}
\nc{\EG}{\bbE_G}
\nc{\EH}{\bbE_H}
\nc{\DEG}{\Der(\EG)}
\nc{\DEH}{\Der(\EH)}
\nc{\DE}{\Der(\bbE)}
\nc{\Prst}{{\cat P}\mathrm{r^{st}}}
\nc{\Mack}[2]{\mathrm{Mack}_{#1}(#2)}
\nc{\SC}{S\cat C}
\dmo{\fin}{{fin}}
\dmo{\DM}{DM}
\dmo{\fp}{fp}
\nc{\DMQ}{\DM_Q}
\dmo{\DerKal}{DMack}
\dmo{\Der}{D}
\dmo{\DMot}{DMot}
\dmo{\rmH}{H}
\dmo{\piu}{\underline{\pi}}
\dmo{\Sphere}{\mathbb{S}}
\nc{\HA}{{\rmH \hspace{-0.2em}\bbA}}
\nc{\HZ}{{\rmH \hspace{-0.2em}\bbZ}}
\nc{\HZbar}{{\rmH \hspace{-0.2em}\underline{\bbZ}}}
\nc{\Fp}{{\bbF_{\hspace{-0.1em}p}}}
\nc{\HFp}{{\rmH \hspace{-0.15em}\bbF_{\hspace{-0.1em}p}}}
\nc{\DHZG}{\Der(\HZ_G)}
\nc{\DHZH}{\Der(\HZ_H)}
\nc{\DHZK}{\Der(\HZ_K)}
\nc{\DHZGN}{\Der(\HZ_{G/N})}
\nc{\DHZGG}{\Der(\HZ_{G/G})}
\nc{\DHZCp}{\Der(\HZ_{C_p})}
\nc{\DHZGprime}{\Der(\HZ_{G'})}
\nc{\DHZ}{\Der(\HZ)}
\nc{\frakp}{\mathfrak{p}}
\nc{\frakq}{\mathfrak{q}}
\nc{\Z}{\mathbb{Z}}
\nc{\F}{\mathbb{F}}
\nc{\SSG}{\text{sSet}_*^G}
\nc{\sSet}{\text{sSet}}

\dmo{\csupp}{supp_{coh}}
\dmo{\Id}{Id}
\dmo{\rmK}{\textrm{\rm K}}
\dmo{\Spc}{Spc}
\dmo{\thick}{thick}
\dmo{\Thick}{Thick}
\nc{\Thicks}[1]{\Thick\langle #1 \rangle}
\nc{\Thickid}[1]{\Thickideal\langle #1 \rangle}
\dmo{\cone}{cone}
\dmo{\End}{End}
\dmo{\Mor}{Mor}
\dmo{\Hom}{Hom}
\dmo{\id}{id}
\dmo{\incl}{incl}
\dmo{\Img}{Im}
\dmo{\im}{im}
\dmo{\Ker}{Ker}
\dmo{\ind}{ind}
\dmo{\CoInd}{coind}
\dmo{\res}{res}
\dmo{\infl}{infl}
\dmo{\triv}{triv}
\dmo{\Tel}{Tel} 
\dmo{\grMod}{grMod}%
\dmo{\Mod}{Mod}%
\dmo{\opname}{op}
\dmo{\SH}{SH}
\dmo{\smallb}{b}
\dmo{\Spec}{Spec}
\dmo{\supp}{supp}
\dmo{\Supp}{Supp}
\nc{\Supph}{\Supp^h}
\nc{\Supphnaive}{\Supp^n}
\dmo{\cosupp}{cosupp}
\dmo{\Cosupp}{Cosupp}
\nc{\Cosupph}{\Cosupp^h}
\nc{\SHc}{{\SH^c}}
\nc{\SHp}{{\SH_{(p)}}}
\nc{\SHcp}{{\SH^c_{(p)}}}
\nc{\SHG}{\SH(G)}
\nc{\SHGp}{\SH(G)_{(p)}}
\nc{\SHGc}{\SHG^c}
\nc{\SHGcp}{\SHG^c_{(p)}}
\nc{\quadtext}[1]{\quad\textrm{#1}\quad}
\nc{\qquadtext}[1]{\qquad\textrm{#1}\qquad}
\nc{\adj}{\dashv}
\nc{\adjto}{\rightleftarrows}
\nc{\bbL}{\mathbb{L}}
\nc{\bbA}{\mathbb{A}}
\nc{\bbE}{\mathbb{E}}
\nc{\bbN}{\mathbb{N}}
\nc{\bbQ}{\mathbb{Q}}
\nc{\bbZ}{\mathbb{Z}}
\nc{\bbF}{\mathbb{F}}
\nc{\cat}[1]{\mathscr{#1}}
\nc{\ie}{{\sl i.e.}, }
\nc{\into}{\mathop{\rightarrowtail}}
\nc{\inv}{^{-1}}
\nc{\isoto}{\mathop{\overset{\sim}\to}}
\nc{\isotoo}{\mathop{\overset{\sim}\too}}
\nc{\onto}{\mathop{\twoheadrightarrow}}
\nc{\too}{\mathop{\longrightarrow}\limits}
\nc{\mapstoo}{\longmapsto}
\nc{\adh}[1]{\overline{#1}}
\nc{\adhpt}[1]{\adh{\{#1\}}}
\nc{\aka}{{a.\,k.\,a.}\ }
\nc{\calF}{\mathcal{F}}
\nc{\eg}{{\sl e.\,g.}}
\nc{\Homcat}[1]{\Hom_{\cat #1}}
\nc{\hook}{\hookrightarrow}
\nc{\ideal}[1]{\langle #1\rangle}
\nc{\ihom}[1]{\mathsf{hom}(#1)}
\nc{\ihomT}[1]{\mathsf{hom}_{\cat T}(#1)}
\nc{\ihomSi}[1]{\mathsf{hom}_{{\cat S}_i}(#1)}
\nc{\Mid}{\,\big|\,}
\nc{\MMod}{\,\text{-}\Mod}%
\nc{\op}{^{\opname}}
\nc{\oto}[1]{\overset{#1}\to}
\nc{\otoo}[1]{\overset{#1}{\,\too\,}}
\nc{\sminus}{\!\smallsetminus\!}
\nc{\poplus}[1]{^{\oplus #1}}%
\nc{\potimes}[1]{^{\otimes #1}}
\nc{\sbull}{{\scriptscriptstyle\bullet}}
\nc{\SET}[2]{\big\{\,#1\Mid#2\,\big\}}
\nc{\SpcK}{\Spc(\cat K)}
\nc{\then}{\Rightarrow}
\nc{\unit}{\mathbb{1}}
\nc{\unitS}{\unit_{\cat S}}
\nc{\unitT}{\unit_{\cat T}}
\nc{\xra}{\xrightarrow}
\nc{\phigeom}[1]{\widetilde{\Phi}^{#1}}
\nc{\phigeomb}[1]{\Phi^{#1}}
\dmo{\Oname}{O}
\dmo{\proper}{proper}
\dmo{\lenormal}{\unlhd}
\dmo{\lnormal}{\lhd}
\nc{\normal}{\trianglelefteq}
\nc{\Op}{\Oname^p}
\nc{\Oq}{\Oname^q}
\dmo{\Sp}{Sp}
\dmo{\Ho}{Ho}
\dmo{\Fin}{Fin}
\dmo{\add}{add}
\dmo{\Fun}{Fun}
\dmo{\Ext}{Ext}
\dmo{\CAlg}{CAlg}
\dmo{\CMon}{CMon}
\dmo{\CC}{\cat C}
\dmo{\DD}{\cat D}
\dmo{\OO}{\mathcal{O}}
\dmo{\Map}{Map}
\dmo{\Span}{Span}
\dmo{\N}{N}
\dmo{\Cat}{Cat}
\dmo{\colim}{colim}
\dmo{\hocolim}{hocolim}
\dmo{\holim}{holim}
\dmo{\Ch}{Ch}
\dmo{\A}{\mathbb{A}^{eff}}
\nc{\AGeff}{\mathbb{A}_G^{\mathrm{eff}}}
\nc{\BGeff}{\mathcal{B}_G^{\mathrm{eff}}}
\nc{\BG}{{\mathcal{B}_G}}
\nc{\NBGeff}{{\N}{\BGeff}}
\dmo{\Ab}{Ab}
\dmo{\Set}{Set}
\dmo{\ev}{ev}
\dmo{\Spcl}{Spcl}
\nc{\Funadd}{\Fun_{\add}}
\dmo{\proj}{proj}
\dmo{\cof}{cof}

\dmo{\Coideal}{Coideal}
\dmo{\gen}{gen}
\dmo{\StMod}{StMod}

\dmo{\projmod}{Lat}
\dmo{\lat}{lat}
\dmo{\Lat}{Lat}
\dmo{\rep}{rep}
\dmo{\Rep}{Rep}
\dmo{\Perf}{Perf}
\dmo{\stmod}{stmod}
\dmo{\Ind}{Ind}
\nc{\borel}[1]{\underline{#1}}
\dmo{\coind}{coind}
\dmo{\rank}{rank}
\nc{\tH}{\hat{H}}
\dmo{\Nm}{Nm}
\dmo{\Proj}{Proj}
\dmo{\Inj}{Inj}
\dmo{\dual}{dual}
\dmo{\fg}{fg}
\nc{\cdvr}[2]{{#1}_{#2}^{\wedge}}
\nc{\cA}{\mathcal{A}}
\dmo{\orbit}{Or}

\nc{\mT}{\kern-0.5em\mod\kern-0.1em\text{-}\cat{T}^c}
\nc{\MT}{\Mod\kern-0.1em\text{-}\cat{T}}
\newcounter{enum-resume-hack}

\dmo{\bP}{\mathbb{P}}
\dmo{\bT}{\mathbb{T}}
\nc{\LOCO}{\mathcal{L}\mathrm{oc}_{\otimes}}

\dmo{\Ob}{Ob}
\nc{\cP}{\mathcal{P}}

\begin{document}


\title[Homological stratification and descent]{Homological stratification and descent}

\author{Tobias Barthel}
\author{Drew Heard}
\author{Beren Sanders}
\author{Changhan Zou}

\date{\today}

\makeatletter
\patchcmd{\@setaddresses}{\indent}{\noindent}{}{}
\patchcmd{\@setaddresses}{\indent}{\noindent}{}{}
\patchcmd{\@setaddresses}{\indent}{\noindent}{}{}
\patchcmd{\@setaddresses}{\indent}{\noindent}{}{}
\makeatother

\address{Tobias Barthel, Max Planck Institute for Mathematics, Vivatsgasse 7, 53111 Bonn, Germany}
\email{tbarthel@mpim-bonn.mpg.de}
\urladdr{\href{https://sites.google.com/view/tobiasbarthel/home}{https://sites.google.com/view/tobiasbarthel/home}}

\address{Drew Heard, Department of Mathematical Sciences, Norwegian University of Science and Technology, Trondheim}
\email{drew.k.heard@ntnu.no}
\urladdr{\href{https://folk.ntnu.no/drewkh/}{https://folk.ntnu.no/drewkh/}}

\address{Beren Sanders, Mathematics Department, UC Santa Cruz, 95064 CA, USA}
\email{beren@ucsc.edu}
\urladdr{\href{https://people.ucsc.edu/~beren/}{http://people.ucsc.edu/$\sim$beren/}}

\address{Changhan Zou, Mathematics Department, UC Santa Cruz, 95064 CA, USA}
\email{czou3@ucsc.edu}
\urladdr{\href{https://people.ucsc.edu/~czou3/}{http://people.ucsc.edu/$\sim$czou3/}}

\begin{abstract}
We introduce a notion of stratification  for rigidly-compactly generated tensor-triangulated categories relative to the homological spectrum and develop the fundamental  features of this theory. In particular, we demonstrate that it exhibits excellent descent properties. In conjunction with Balmer's Nerves of Steel conjecture, we  conclude  that classical stratification also admits a general form of descent. This gives a uniform  treatment of several recent stratification results and provides a complete answer to the question: \textit{When does stratification descend?}  As a new application, we extend earlier work  on the  tensor triangular geometry  of equivariant module spectra from finite groups to compact Lie groups.
\end{abstract}


\maketitle
{
\hypersetup{linkcolor=black}
\tableofcontents
}

\section*{Introduction}\label{sec:introduction}

Tensor triangular geometry, classically construed, studies the global and local geometry of a \emph{small} tensor-triangulated category $\cat K$  through its spectrum $\Spc(\cat K)$. This topological space, introduced by Balmer~\cite{Balmer05a}, comes equipped with the universal theory of support for the objects of $\cat K$, and yields a classification of the thick ideals of $\cat K$ in terms of certain subsets of~$\Spc(\cat K)$, the so-called Thomason subsets. An important direction in recent years has been to extend this theory to the context of `big'\footnote{that is, rigidly-compactly generated; see below for details.} tensor-triangulated categories~$\cat T$; here, $\cat K$ arises as the full subcategory $\cat T^c$ of compact-dualizable objects in $\cat T$.

En route to the computation of the spectrum of the  compact-dualizable objects in the derived category $\Der(R)$ of a noetherian commutative ring $R$, Neeman~\cite{Neeman92a}  also classifies the localizing ideals of $\Der(R)$ by showing that they correspond to arbitrary subsets of $\Spec(R)$. A systematic approach to such classification results for big tensor-triangulated categories was then introduced by Hovey, Palmieri and Strickland \cite{HoveyPalmieriStrickland97} and further developed by Benson, Iyengar and Krause \cite{BensonIyengarKrause08,BensonIyengarKrause11b}. The resulting notion of \emph{cohomological stratification} for a big tensor-triangulated category $\cat T$ relies on the action of a noetherian ring $R$ on~$\cat T$ in order to construct a suitable support theory. Cohomological stratification then asserts that this support theory induces a bijection 
    \[
        \Supp_R\colon \big\{ \text{localizing ideals of $\cat T$} \big\} \xra{\cong} \big\{ \text{subsets of $\Spec(R)$}\big\}.
    \]
In practice, $R$ is often taken to be the (graded) endomorphism ring $\End_{\cat T}^{\bullet}(\unit)$ of the unit in $\cat T$.\footnote{More generally, cohomological stratification operates in the setting of arbitrary compactly generated triangulated categories, i.e., it does not require a tensor structure on $\cat T$.} A prominent  application of this theory is the stratification of the stable module category of a finite group in modular characteristic  due to Benson, Iyengar and Krause  \cite{BensonIyengarKrause11a}. Key to their proof is a descent result for cohomological stratification that provides a reduction from finite groups to elementary abelian groups, which can then be tackled via Neeman's theorem. 

A significant drawback of cohomological stratification is that it can only apply in situations where the parametrizing object is \emph{affine} and \emph{noetherian}. Addressing both issues simultaneously, in \cite{bhs1} the three first-named authors introduced a notion of \emph{tensor-triangular stratification} which is instead based on the Balmer--Favi notion of support \cite{BalmerFavi11}. The latter takes values in  $\Spc(\cat T^c)$ and tensor-triangular stratification asserts  that it provides  a bijection
    \[
        \Supp\colon \big\{ \text{localizing ideals of $\cat T$} \big\} \xra{\cong} \big\{ \text{subsets of $\Spc(\cat T^c)$}\big\}.
    \]
	The resulting theory has both theoretical and practical advantages, relates to \mbox{cohomological} stratification through Balmer's comparison map 
    \[
        \rho\colon \Spc(\cat T^c) \to \Spec(\End_{\cat T}^{\bullet}(\unit))
    \]
and applies to numerous new examples. However, two  serious issues remain: 
Firstly, it still requires a topological assumption on the spectrum, namely the hypothesis of being ``weakly noetherian''.
Secondly, while the permanence of stratification has been established via  many instances of descent (e.g., Zariski descent, finite \'etale descent, nil descent, etc; see \cite{bhs1,BCHS1,BCHNPS_descent}), a general descent statement remains elusive. 

In this paper, we introduce a theory of stratification  based on the ``homological'' support theory $\Supph$ introduced in \cite{Balmer20_bigsupport} which takes values in the \emph{homological spectrum} $\smash{\Spc^h(\cat T^c)}$. This is a variant of the usual spectrum (also introduced by Balmer \cite{Balmer20_nilpotence}) whose points  correspond to the \emph{homological residue fields} of $\cat T^c$. The key features of our theory of \emph{homological stratification} may be summarized informally as follows:

\begin{itemize}
    \item (Generality) Homological stratification works without any point-set topological restrictions on the spectrum, thus avoiding the (weakly) noetherian assumptions required for tensor-triangular and cohomological stratification. 
    \item (Descendability) Homological stratification satisfies a very general form of descent. This recovers, unifies and extends  all known descent results for stratification in the literature.
    \item (Refinement) Homological stratification in general refines the notion of tensor-triangular stratification and coincides with it in cases where Balmer's Nerves of Steel conjecture holds.
\end{itemize}

The last point requires some explanation. There is a canonical continuous and surjective  map
\begin{equation}\label{eq:intro-pi}\tag{$\dagger$}
        \begin{tikzcd}
		\Spc^h(\cat T^c) \ar[r,->>,"\pi"] &\Spc(\cat T^c)
        \end{tikzcd}
	\end{equation}
and the Nerves of Steel conjecture states that $\pi$ is a bijection \cite{Balmer20_nilpotence}. It is known  to hold in numerous examples. When this is the case, the homological theory of stratification coincides with its tensor-triangular counterpart.  On the other hand,  any counterexample to the Nerves of Steel conjecture would have  the property that the homological spectrum contains more points than the usual one, so homological stratification would have access to  more refined information about~$\cat T$.

\begin{figure}[h!]
    \[
    \xymatrix@R=3em{
        \fbox{\text{\parbox{0.8in}{\centering homological stratification}}} \ar@{~>}[r] \ar[d] &\text{\parbox{1.2in}{\centering tensor-triangular stratification}} \ar@{~>}[r] \ar[d]& \text{\parbox{1in}{\centering cohomological stratification}} \ar[d] \\
        \Spc^h(\cat T^c) \ar@{->>}[r]^-{\pi} &  \Spc(\cat T^c) \ar[r]^-{\rho} & \Spec(\End_{\cat T}^{\bullet}(\unit))
    }
    \]
    \caption{Schematic overview of the relations between the various notions of stratification.}\label{fig:overview}
\end{figure}
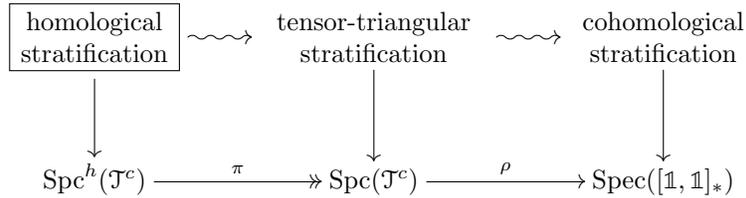

The rich interplay between the notions of homological and tensor-triangular stratification is one of the themes of this paper. Each theory has certain advantages over the other; up to the Nerves of Steel conjecture, we can combine both features. Collecting several of our results, we can characterize when tensor-triangular stratification descends along a jointly conservative geometric family of functors, under the assumption that the Nerves of Steel conjecture holds: 

\begin{thmx}\label{thmx:a}
    Let $(f_i^*\colon \cat T \to \cat S_i)_{i \in I}$ be a family of geometric functors that jointly detect when an object of $\cat T$ is zero. Suppose that $\cat S_i$ is tt-stratified for all $i \in I$. If~$\cat T$ satisfies the Nerves of Steel conjecture and has a weakly noetherian spectrum, then the following conditions are equivalent:
        \begin{enumerate}
            \item $\cat T$ is tt-stratified;
            \item $\cat T$ is h-stratified;
            \item $\cat T$ is generated by the images of the right adjoints $(f_i)_*$ for all $i \in I$.
        \end{enumerate}
\end{thmx}

\noindent In other words, this result provides one answer to the question: 
    \[
    \emph{When does stratification descend?}
    \]
Its proof is assembled at the end of \cref{sec:comp-strat}. In addition, as we will demonstrate, it encompasses all known descent results for stratification and also gives rise to new ones.

\subsection*{Main results}\label{ssec:mainresults}

We now give a more detailed overview of the main results of  the paper. Throughout, $\cat T = (\cat T, \otimes, \mathsf{hom}, \unit)$ denotes a rigidly-compactly generated  tensor-triangulated category, that is,  a tensor-triangulated (``tt'') category  that  is compactly generated as a triangulated category and has the property that the compact objects and the dualizable objects in $\cat T$ coincide: $\cat T^c=\cat T^d$. A geometric functor $f^*\colon \cat T \to \cat S$ between such categories is  an exact and symmetric monoidal functor  that preserves arbitrary set-indexed coproducts. In particular, $f^*$ admits a right adjoint $f_*$ which  itself admits a right adjoint $f^!$.

We say that $\cat T$ is \emph{homologically stratified} (or \emph{h-stratified}, for short) if homological support induces a bijection
    \[ 
        \Supph\colon \big\{ \text{localizing ideals of $\cat T$} \big\} \xra{\cong} \big\{ \text{subsets of $\Spc^h(\cat T^c)$}\big\}.
    \]
We  emphasize that  we are making no point-set topological assumptions on the homological spectrum; see \cref{def:hstratification}. Our first result provides a concrete  criterion for  homological stratification:

\begin{thmx}[\cref{thm:hstratfundamental} and \cref{thm:hstratcosupp}]
    For a rigidly-compactly generated tt-category $\cat T$, the following are equivalent:
    \begin{enumerate}
        \item $\cat T$ is homologically stratified;
        \item $\cat T$ satisfies the following two conditions:
			\begin{enumerate}[label=(\roman*)]
                \item the homological local-to-global principle holds,
					that is,
            \[  
                \Loco{t} = \Loco{t \otimes E_{\cat B} \mid \cat B \in \Supph(t)}
            \]
					for all $t \in \cat T$;
        \item $\Loco{E_{\cat B}}$ is a minimal localizing ideal for all $\cat B \in \Spc^h(\cat T^c)$;
            \end{enumerate}
        \item $\ihom{t_1,t_2} = 0$ if and only if $\Supph(t_1) \cap \Cosupph(t_2) = \emptyset$ for all $t_1,t_2 \in \cat T$.
    \end{enumerate}
\end{thmx}

Here $E_{\cat B}$ is a pure-injective object in $\cat T$ associated to each  $\cat B \in \Spc^h(\cat T^c)$; see \cref{rem:support-and-cosupports}. The equivalence of the first two conditions  is completely analogous to the characterization of cohomological stratification and tt-stratification established in \cite[Theorem 4.2]{BensonIyengarKrause11b} and \cite[Theorem 4.1]{bhs1}, respectively. The third characterization is new and involves a notion of homological cosupport $\Cosupp^h$ which we introduce and study in this paper, motivated by its tt-theoretic counterpart studied in \cite{BCHS1}.

When establishing that a given category is tt-stratified by checking the conditions analogous to part (b) of the above theorem, the bulk of the work goes into establishing minimality. In fact, for cohomological stratification --- where the Zariski spectrum is  noetherian --- stratification is equivalent to minimality. In sharp contrast, the homological minimality condition should --- heuristically speaking --- be detected in the residue fields of $\cat T$ and thereby be much easier to verify in practice.
This is the underlying reason why the theory of homological stratification tends to be better behaved under descent than its cousins.

For instance, in favourable situations, the points of $\Spc^h(\cat T^c)$ are detected not only by homological residue fields, but by tt-residue fields of $\cat T$, i.e., for each $\cat B \in \Spc^h(\cat T)$ there exists a geometric functor $\cat T \to \cat F$ to a tt-field  which detects $\cat B$. If this is the case, we say that $\cat T$ \emph{admits enough tt-fields} (\cref{def:enough-tt-fields}). As demonstrated in \cite{BalmerCameron2021}, examples abound. Under the assumption that $\cat T$ admits enough residue fields, homological stratification reduces to the homological local-to-global principle: 

\begin{thmx}[\cref{cor:ttfields1,prop:ttfields2}]
    Suppose that $\cat T$ admits enough tt-fields $(f_{\cat B}^*\colon \cat T \to \cat F_{\cat B})$. Then the following are equivalent:
        \begin{enumerate}
            \item $\cat T$ is h-stratified;
            \item $\cat T$ satisfies the homological local-to-global principle;
            \item $\cat T$ is generated as a localizing ideal by the $(f_{\cat B})_*(\cat F_{\cat B})$.
        \end{enumerate}
\end{thmx}

The last condition  says that the homological local-to-global principle holds in $\cat T$ whenever its tt-residue fields cover the entire category. It is possible to substantially generalize this to a relative context. To this end, we isolate a very general descent condition: A collection of geometric functors $(f_i^*\colon \cat T \to \cat S_i)_{i\in I}$ is said to be \emph{weakly descendable} (\cref{def:weaklydescendable}) if the essential images of the corresponding right adjoints $(f_i)_*$ generate $\cat T$ as a localizing ideal:
    \[
        \cat T = \Loco{(f_i)_*(\unit_{\cat S_i})}.
    \]
Note that we do not  require any smallness conditions on $(f_i)_*(\unit_{\cat S_i}) \in \cat T$. This notion of weak descendability covers all examples of descent previously considered in the tt-literature, and is arguably the most general setting in which one could expect descent to hold (\cref{rem:weak-descendable-best}). With this in mind, the following result establishes that h-stratification \emph{essentially always} satisfies descent.

\begin{thmx}[\cref{thm:h-descent}]
    Let $(f_i^*\colon \cat T \to \cat S_i)_{i \in I}$ be a weakly descendable family of geometric functors. If $\cat S_i$ is h-stratified for all $i \in I$, then $\cat T$ is h-stratified. 
\end{thmx}

This is in marked contrast to tt-stratification, where  one has so far only been able to prove a patchwork of different descent theorems. With a satisfactory descent  theorem for homological stratification in hand, one may then ask for a comparison with  tt-stratification:

\begin{thmx}[\cref{thm:tt=h+NS}]\label{thmx:tt=h+ns}
    If $\cat T$ is a rigidly-compactly generated tt-category with $\Spc(\cat T^c)$ weakly noetherian, then the following are equivalent:
        \begin{enumerate}
            \item $\cat T$ is tt-stratified;
            \item $\cat T$ is h-stratified and the Nerves of Steel conjecture holds for $\cat T$.
        \end{enumerate}
\end{thmx}

Thus, whenever the Nerves of Steel conjecture is known or can be established through different  methods, it suffices to consider the theory of homological stratification. This enables us to deduce strong novel descent results for  tt-stratification itself. For instance, we prove:

\begin{thmx}[\cref{prop:ttstratdescent}$(c)$ and \cref{rem:ttdescent-finite}]
    Let $(f_i^*\colon \cat T \to \cat S_i)_{i \in I}$ be a weakly descendable geometric family with each $\cat S_i$ tt-stratified. Assume additionally that $I$ is finite, that each right adjoint $(f_i)_*$ preserves compact objects, and that each $\Spc(\cat S_i^c)$ is noetherian. Then $\cat T$ is tt-stratified. 
\end{thmx}

This result generalizes the \'etale descent theorem from \cite[Theorem 6.4]{bhs1} and \cite[Section 2.2.2]{barthel2022rep2} by removing the separability condition present there.  As a special case, we obtain the following result, which generalizes the descent theorem of \cite[Example 12.18]{BCHNPS_descent}:

\begin{thmx}[\cref{cor:ttstratfinitemonodescent}]
    Let $\cat C$ be a rigidly-compactly generated symmetric monoidal stable $\infty$-category and let $A \in \CAlg(\cat C^c)$ be a descendable dualizable commutative algebra in $\cat C$. If $\Mod_{\cat C}(A)$  is tt-stratified and has a noetherian spectrum, then $\cat C$ is also tt-stratified and also has a noetherian spectrum.
\end{thmx}

We also generalize nil-descent and quasi-finite descent established in \cite{BCHS1} to families of geometric functors:

\begin{thmx}[\cref{prop:ttstratdescent}$(a)(b)$]
   Let $(f_i^*\colon\cat T \to \cat S_i)_{i \in I}$ be a weakly descendable family of geometric functors with $\cat S_i$ tt-stratified for all $i \in I$ and let
    \[
    \varphi \coloneqq \sqcup \varphi_i\colon \bigsqcup_{i \in I}\Spc(\cat S_i^c) \to \Spc(\cat T^c)
    \]
    be the induced map on tt-spectra. Assume that $\Spc(\cat T^c)$ is weakly noetherian. Then:
   \begin{enumerate}
       \item (Nil-descent) If $\varphi$ is injective then $\cat T$ is tt-stratified.
       \item (Quasi-finite descent) If each right adjoint $(f_i)_*$ preserves compact objects and $\varphi$ has discrete fibers then $\cat T$ is tt-stratified.
   \end{enumerate}
\end{thmx}
As a novel application of our results, we prove the following:
\begin{thmx}[\cref{thm:infleq_stratification}]\label{thm:intro-equivariant}
    Let $G$ be a compact Lie group and write $R_G$ for the inflation of a commutative ring spectrum $R\in \CAlg(\Sp)$. If $\Mod(R)$ is h-stratified and satisfies the Nerves of Steel conjecture, then $\pi\circ\Supph$ induces a bijection
        \[
            \begin{Bmatrix}
                \text{localizing ideals} \\
                of \Mod_G(R_G)
            \end{Bmatrix}
                \xrightarrow{\sim}
            \begin{Bmatrix}
                \text{subsets of } \\
                \bigsqcup_{H \in \Sub(G)/G} \Spc(\Mod(R)^c)
            \end{Bmatrix}.
        \]
\end{thmx}
This extends \cite[Theorem 15.1]{bhs1} from finite groups to compact Lie groups. Indeed, there we proved the finite analog of \Cref{thm:intro-equivariant} under the assumption that $\Mod(R)$ was tt-stratified, which implies both the Nerves of Steel conjecture and h-stratification by \Cref{thmx:tt=h+ns}. 

\subsection*{Outline of the document.}\label{ssec:outline}
In \cref{sec:hsupport}, we recall the definition and basic  features of  homological support and introduce the homological detection property. We also discuss the ``naive'' definition of homological support and establish situations in which it coincides with the  homological support. Finally, we  introduce  homological cosupport and the homological codetection property. In \cref{sec:hLGP}, we introduce the homological local-to-global principle and discuss its relationship with the \mbox{h-codetection} property. This leads to the definition of h-stratification in \cref{sec:hstrat} and the fundamental characterization of h-stratification in terms of the homological local-to-global principle and homological minimality. We discuss examples and counterexamples arising in commutative algebra and explain that  h-stratification is equivalent to the  homological local-to-global principle when there are enough tt-fields (\cref{cor:ttfields1}). We also show  that a classification of Bousfield classes is implied by (and almost equivalent to) the h-detection property (\cref{prop:bousfield-classes}). This surprising result has no tt-analogue. Finally, in \cref{sec:hstrat-cosupp}, we provide a further characterization of h-stratification in terms of h-cosupport (\cref{thm:hstratcosupp}).

We compare the homological support with the Balmer--Favi support in \cref{sec:comparison} and establish several base change results for the homological support in \cref{sec:base-change}. Of particular note is \cref{cor:AV-scott}, which establishes that the homological support satisfies the Avrunin--Scott identity whenever the h-detection property holds. In \cref{sec:weakly-descendable}, we introduce the notion of a weakly descendable family of functors and argue that this is the appropriate  context  in which to study descent. This culminates in \cref{thm:h-descent}. We then turn to the comparison between h-stratification and tt-stratification in \cref{sec:comp-strat} yielding \cref{thm:tt=h+NS} and its corollaries. Then in \cref{sec:descendingNSC} we turn to the task of descending the Nerves of Steel conjecture, yielding \cref{prop:NSCdescent} and \cref{prop:nil-cons-NSC}. We also show that the Nerves of Steel conjecture is equivalent to a purely support-theoretic statement (\cref{prop:tensorformulaweakrings}). Our descent results are then applied in \cref{sec:tt-applications} to obtain several specific descent results for tt-stratification, including \cref{cor:ttstratfinitemonodescent} among others. Applications to equivariant spectra are then given in \cref{sec:applications}, including \cref{thm:infleq_stratification}. Finally, we conclude the paper in \cref{sec:open-questions} with a list of open problems.

\subsection*{Notation and conventions.}\label{ssec:conventions}
Throughout $\cat T$ and $\cat S$ will denote rigidly-compactly generated tt-categories and $f^*\colon \cat T\to \cat S$ will denote a geometric functor.
\begin{itemize}
	\item We write $\varphi\colon\Spc(\cat S^c)\to\Spc(\cat T^c)$ for the induced map on tt-spectra and write $\varphi^h\colon\Spc^h(\cat S^c) \to \Spc^h(\cat T^c)$ for the induced map on homological spectra.
	\item We write $\pi_{\cat T}\colon \Spc^h(\cat T^c) \to \Spc(\cat T^c)$ for the map between the two spectra.
	\item We thus have a commutative diagram
		\[\begin{tikzcd}
			\Spc^h(\cat S^c) \ar[d,"\pi_{\cat S}"] \ar[r,"\varphi^h"] & \Spc^h(\cat T^c) \ar[d,"\pi_{\cat T}"] \\
			\Spc(\cat S^c) \ar[r,"\varphi"] & \Spc(\cat T^c).
		\end{tikzcd}\]
	\item We denote the right adjoint of $f^*$ by $f_*$ and the right adjoint of $f_*$ by~$f^!$;  in symbols, $f^* \dashv f_* \dashv f^!$.  We will use the standard isomorphisms relating these functors established in \cite{BalmerDellAmbrogioSanders16} without further comment. 
    \item We write $\Loco{\cat E}$ and $\Coloco{\cat E}$ for the localizing ideal and the colocalizing coideal generated by a collection of objects $\cat E \subseteq \cat T$, respectively. 
    \item Although not necessary to understand the logic of the results,  some familiarity with the tt-theory of stratification and costratification developed in \cite{bhs1} and \cite{BCHS1} and the descent results of \cite{BCHNPS_descent} will be helpful to contextualize the results of this paper.
    \item We also take for granted some familiarity with the homological spectrum; see \cite{Balmer20_nilpotence,Balmer20_bigsupport}.
\end{itemize}

\subsection*{Acknowledgments}\label{ssec:thanks}

We thank Scott Balchin and Paul Balmer for helpful discussions, as well as Jordan Williamson and the anonymous referee for useful comments on an earlier version. TB is supported by the European Research Council (ERC) under Horizon Europe (grant No.~101042990) and would like to thank the Max Planck Institute for its hospitality. DH is supported by grant number TMS2020TMT02 from the Trond Mohn Foundation.

\section{Homological support and cosupport}\label{sec:hsupport}

We begin with a discussion on the various notions of homological (co)support considered in this paper. 

\begin{Rem}\label{rem:support-and-cosupports}
	Recall from \cite[Construction 2.11]{Balmer20_bigsupport} that there is a pure-injective object $E_{\cat B} \in \cat T$ associated to each homological prime $\cat B \in \Spc^h(\cat T^c)$ and that~$E_{\cat B}$ has the structure of a weak ring in $\cat T$. For each object $t \in \cat T$, we define
    \begin{itemize}
	\item The \emph{naive homological support}
		\[\Supphnaive(t) \coloneqq \SET{\cat B \in \Spc^h(\cat T^c)}{E_{\cat B} \otimes t\neq 0}.\]
	\item The \emph{(genuine) homological support} 
		\[\Supph(t) \coloneqq \SET{\cat B \in \Spc^h(\cat T^c)}{\ihom{t,E_{\cat B}}\neq 0}.\]
	\item The \emph{homological cosupport}
		\[ \Cosupph(t) \coloneqq \SET{\cat B \in \Spc^h(\cat T^c)}{\ihom{E_{\cat B},t} \neq 0}.\]
    \end{itemize}
    The homological cosupport does not appear to have been previously studied in this level of generality, but it will play its own role in the story, as in \cite{BCHS1}.
\end{Rem}

\begin{Exa}\label{exa:suppEb}
    For each $\cat B \in \Spc^h(\cat T)$, we have
        \[
            \Supph(E_{\cat B}) = \Supphnaive(E_{\cat B}) = \Cosupph(E_{\cat B}) = \{\cat B\}
        \]
    by \cite[Example 4.8 and Corollary 4.9]{Balmer20_bigsupport}.
\end{Exa}

\begin{Rem}
	The homological support was introduced by Balmer \cite{Balmer20_bigsupport} and in that paper he establishes a number of significant properties for it. For example, \cite[Theorem 4.5]{Balmer20_bigsupport} shows that the (genuine) homological support satisfies the \emph{tensor product property}:
	\begin{equation}\label{eq:tensor-product}
        \Supph(t_1 \otimes t_2) = \Supph(t_1) \cap \Supph(t_2)
	\end{equation}
    for all $t_1,t_2\in \cat T$. He also proves the following:
\end{Rem}

\begin{Prop}\label{prop:supph-in-suppn}
	For any $t \in \cat T$, we have
    \[
        \Supph(t) \subseteq \Supphnaive(t).
    \]
\end{Prop}

\begin{proof}
    As pointed out in \cite{Balmer20_bigsupport}, this follows from the tensor product property and the fact that $\Supph(E_{\cat B})=\{\cat B\}$. On the other hand, part $(b)$ of \Cref{lem:formal1} below (applied to the weak ring $E_{\cat B}$) provides a more ``formal'' proof of the result.
\end{proof}

\begin{Lem}\label{lem:formal1} The following hold:
	\begin{enumerate}
		\item If $C$ is a weak coring then
			\[
                C \otimes t = 0 \Longrightarrow \ihom{C,t} = 0.
            \]
		\item If $R$ is a weak ring then
			\[
                t \otimes R = 0 \Longrightarrow \ihom{t,R} = 0.
            \]
	\end{enumerate}
\end{Lem}

\begin{proof}
	$(a)$: Let $\epsilon\colon C \to \unit$ be the counit of the weak coring structure and let $\Delta\colon C \to C\otimes C$ be a choice of ``comultiplication'',  that is,  a section of the  map $1\otimes \epsilon\colon {C \otimes C \to C}$. One can check using dinaturality of coevaluation  (with respect to $\Delta$) that the composite
	\[
        \ihom{C,t} \xrightarrow{1 \otimes \mathrm{coev}} \ihom{C,t}\otimes \ihom{C,C} \xrightarrow{} \ihom{C\otimes C, t\otimes C} \xrightarrow{\ihom{\Delta,1\otimes\epsilon}} \ihom{C,t}
    \]
	is the identity.

	$(b)$: Let $\eta\colon\unit \to R$ be the unit of the weak ring structure  and let $\mu:R\otimes R \to R$ be a choice of ``multiplication'',  that is, a retraction of the map $1\otimes \eta\colon R \to R \otimes R$. One can check using dinaturality of coevaluation  (with respect to $1\otimes \eta\colon t \to t\otimes R$) that the composite
	\[
        \ihom{t,R} \xrightarrow{1 \otimes \mathrm{coev}} \ihom{t,R} \otimes \ihom{R,R} \xrightarrow{} \ihom{t\otimes R,R\otimes R} \xrightarrow{\ihom{1\otimes\eta,\mu}} \ihom{t,R}
    \]
	is the identity.
\end{proof}

\begin{Lem}\label{lem:formal2}
	If $C$ is a weak coring then 
	\[
        \ihom{C,t} = 0 \Longrightarrow C \otimes t = 0.
    \]
\end{Lem}

\begin{proof}
	Let $\epsilon\colon C\to \unit$ be the counit. One can readily check using dinaturality of evaluation (with respect to $\epsilon$) that the composite
    \[
        t \otimes C \simeq \ihom{\unit,t}\otimes C \xrightarrow{\ihom{\epsilon,1}\otimes 1} \ihom{C,t}\otimes C \xrightarrow{\mathrm{ev}} t \simeq t \otimes \unit
    \]
    is $1 \otimes \epsilon\colon t \otimes C \to t \otimes \unit$. Thus, $\ihom{C,t}=0$ implies $t \otimes \epsilon = 0$. Since the identity of $t\otimes C$ factors through $t \otimes C \otimes C \xrightarrow{1 \otimes \epsilon \otimes 1} t \otimes C$, we conclude that $t\otimes C=0$.
\end{proof}

\begin{Prop}\label{prop:supp-coincide-weak-coring}
    We have the equality 
	\[
        \Supph(t) = \Supphnaive(t)
    \]
    whenever $t$ is a weak ring or a weak coring.
\end{Prop}

\begin{proof}
    The case when $t$ is a weak ring is \cite[Theorem 4.7]{Balmer20_bigsupport} while the case when~$t$ is a weak coring  follows from \cref{lem:formal2} and part $(a)$ of \cref{lem:formal1}.
\end{proof}

\begin{Rem}\label{rem:hsuppdetectsweakrings}
	Balmer also establishes that if $t \in \cat T$ is a weak ring then ${\Supph(t)=\emptyset}$ implies $t=0$. However, since $\Supph$ satisfies the tensor product property, the existence of categories which have nonzero tensor-nilpotent objects shows that in general it is possible for $\Supph(t) = \emptyset$ even with $t \neq 0$. This leads to:
\end{Rem}

\begin{Def}\label{def:hdetection}
    Let $\cat T$ be a rigidly-compactly generated tt-category.
	\begin{enumerate}
		\item  We say that the \emph{h-detection property} holds for $\cat T$ (short for \emph{homological detection property}) if $\Supph(t) = \emptyset$ implies $t=0$.
        \item  We say that the \emph{h-codetection property} holds for $\cat T$ (short for \emph{homological codetection property})
	if $\Cosupph(t) = \emptyset$ implies $t=0$.
    \end{enumerate}
\end{Def}

\begin{Lem}\label{lem:supp-coincide-h-detection}
	If $\cat T$ has the h-detection property then  $\Supph(t) = \Supphnaive(t)$ for all $t \in \cat T$.
\end{Lem}

\begin{proof}
	We always have the $\subseteq$ inclusion by \cref{prop:supph-in-suppn}. On the other hand, if $\cat B \in \Supphnaive(t)$ then $E_{\cat B} \otimes t \neq 0$ by definition. Hence the h-detection property implies that $\Supph(E_{\cat B} \otimes t) \neq \emptyset$. By the tensor product property this implies $\emptyset \neq \Supph(E_{\cat B}) \cap \Supph(t) = \{\cat B\} \cap \Supph(t)$ so that $\cat B \in \Supph(t)$.
\end{proof}

\begin{Rem}
	As previously mentioned, the h-detection property does not always hold (\cref{rem:hsuppdetectsweakrings}). Nevertheless, we can prove that the two notions of homological support coincide for many additional examples of interest, as follows.
\end{Rem}

\begin{Def}\label{def:enough-tt-fields}
	\hspace{-0.2ex}A rigidly-compactly generated tt-category $\cat T$
	\emph{admits enough \mbox{tt-fields}}
	if for each $\cat B \in \Spc^h(\cat T^c)$
	there exists a geometric functor $\cat T \to \cat F$ to a tt-field, in the sense of \cite[Definition 1.1]{BalmerKrauseStevenson19},
	which maps the unique point in $\Spc^h(\cat F^c)$ to~$\cat B$.
	Note that $\Spc^h(\cat F^c)$ consists of a single point by \cite[Theorem~5.17(a)]{BalmerKrauseStevenson19}.
\end{Def}

\begin{Exa}\label{exa:DR-tt-fields}
	The derived category $\Der(R)$ of a commutative ring $R$ admits enough tt-fields. Indeed, $\Spc^h(\Der(R)^c) \cong \Spec(R)$ and tt-fields are  provided by the usual residue fields $\Der(R) \to \Der(\kappa(\mathfrak p))$. For the homological prime $\cat B\in \Spc^h(\Der(R)^c)$ corresponding to $\mathfrak p \in \Spec(R)$, we have $E_{\cat B} \simeq \kappa(\mathfrak p)$, considered as a complex concentrated in degree~0; see \cite[Corollary 3.3]{BalmerCameron2021}. 
\end{Exa}

\begin{Lem}\label{lem:naive-tt-field}
	Let $f^*\colon \cat T \to \cat F$ be a geometric functor to a tt-field $\cat F$. Let $\cat B \in \Spc^h(\cat T^c)$  be the image of the unique  point in $\Spc^h(\cat F^c)$. Then $\cat B \in \Supphnaive(t)$ if and only if $\cat B \in \Supph(t)$.
\end{Lem}

\begin{proof}
    The inclusion $\Supph(t) \subseteq \Supphnaive(t)$ is \cref{prop:supph-in-suppn}, so it remains to prove the ``only if'' direction. First we claim that $f^! E_{\cat B} \neq 0$. Otherwise, we would have $\ihom{f_*(\unit),E_{\cat B}}=f_*f^! E_{\cat B} =0$. But $E_{\cat B}$ is a direct summand of $f_*(\unit)$ by \cite[Theorem 3.1]{BalmerCameron2021}.  Thus we would have $\ihom{E_{\cat B},E_{\cat B}} =0$ which contradicts $E_{\cat B} \neq 0$.

    If $\ihom{t,E_{\cat B}} = 0$, then $\ihom{f^* t,f^! E_{\cat B}} = f^!\ihom{t,E_{\cat B}} =0$. Since $\cat F$ is a tt-field, it follows that $f^* t = 0$ by \cite[Theorem 5.21]{BalmerKrauseStevenson19}. Hence $f_*(\unit) \otimes t = 0$ by the projection formula. Therefore $E_{\cat B} \otimes t = 0$ since $E_{\cat B}$ is a direct summand of~$f_*(\unit)$.
\end{proof}

\begin{Prop}\label{prop:naive-tt-fields}
	Suppose $\cat T$ admits enough tt-fields (\cref{def:enough-tt-fields}). Then 
    \[
        \Supph(t) = \Supphnaive(t)
    \]
    for all $t \in \cat T$.
\end{Prop}

\begin{proof}
	This follows immediately from \cref{lem:naive-tt-field}. 
\end{proof}

\begin{Par}
	We conclude this section with one further result concerning the relationship between the naive and genuine notions of homological support.
\end{Par}

\begin{Lem}\label{lem:minimalityhsuppidentification}
    Let $\cat B \in \Spc^h(\cat T^c)$ and suppose $\Loco{E_{\cat B}}$ is a minimal localizing ideal in $\cat T$. Then $\cat B \in \Supphnaive(t)$ if and only if $\cat B \in \Supph(t)$.
\end{Lem}

\begin{proof}
	The $(\Rightarrow)$ direction holds unconditionally by \cref{prop:supph-in-suppn}, so it suffices to prove the  $(\Leftarrow)$ implication. Let $t \in \cat T$ with $t \otimes E_{\cat B} \neq 0$. By our minimality assumption, this implies that $E_{\cat B} \in \Loco{t \otimes E_{\cat B}}$. If we had $\ihom{t \otimes E_{\cat B}, E_{\cat B}} = 0$, then $\ihom{E_{\cat B}, E_{\cat B}} = 0$ which contradicts  $E_{\cat B} \neq 0$. Therefore, we see
    \[
        0 \neq \ihom{t \otimes E_{\cat B}, E_{\cat B}} \simeq \ihom{E_{\cat B},\ihom{t,E_{\cat B}}}.
    \]
    In particular, this implies that $\ihom{t,E_{\cat B}} \neq 0$.
\end{proof}

\section{The homological local-to-global principle}\label{sec:hLGP}

\begin{Def}\label{def:hlocaltoglobal}
    We say that $\cat T$ satisfies the \emph{homological local-to-global principle} (or the \emph{h-LGP}, for short) if we have the equality
    \[
        \Loco{t} = \Loco{t \otimes E_{\cat B} \mid \cat B \in \Supph(t)}
    \]
    for all $t \in \cat T$.
\end{Def}

\begin{Prop}\label{prop:hLGP}
	The following are equivalent:
	\begin{enumerate}
		\item $\unit \in \Loco{E_{\cat B} \mid \cat B \in \Spc^h(\cat T^c)}$;
		\item $t \in \Loco{t\otimes E_{\cat B} \mid \cat B \in \Spc^h(\cat T^c)}$ for all $t \in \cat T$;
		\item $t \in \Loco{t\otimes E_{\cat B} \mid \cat B \in \Supphnaive(t)}$ for all $t \in \cat T$;
		\item The h-codetection property holds: $\Cosupph(t) = \emptyset$ implies $t=0$.
	\end{enumerate}
\end{Prop}

\begin{proof}
	The equivalence of $(a)$ and $(d)$ follows verbatim from the argument in \cite[Theorem 6.4]{BCHS1} which establishes the equivalence of the classical LGP property and the classical codetection property. On the other hand, part $(a)$ is equivalent to part $(b)$  by \mbox{\cite[(2.6)]{BCHS1}}. Finally, the equivalence of $(b)$ and $(c)$ is immediate from the definition of the naive h-support.
\end{proof}

\begin{Rem}\label{rem:hLGPannoying}
	The h-LGP property  of \cref{def:hlocaltoglobal} implies the equivalent conditions of \cref{prop:hLGP}. However, \emph{a priori} the latter conditions are weaker since we do not know whether $\Supph(t)=\Supphnaive(t)$ in general. In particular, it is not immediate from the above whether the h-codetection property implies the h-detection property. This is in contrast to the situation in \cite{BCHS1}. However, we have the following conditional result:
\end{Rem}

\begin{Cor}\label{cor:hLGP}
    Suppose that $\cat T$ has the h-codection property. Then the following are equivalent:
        \begin{enumerate}
            \item $\cat T$ satisfies the h-LGP;
            \item $\Supph(t)=\Supphnaive(t)$ for all $t \in \cat T$.
        \end{enumerate}
\end{Cor}

\begin{proof}
    Suppose that $(a)$ holds, so that in particular h-detection holds. Indeed, $\Supph(t) = \varnothing$ implies $\Loco{t} = 0$ by the h-LGP, hence $t = 0$. It then follows from \Cref{lem:supp-coincide-h-detection} that $(b)$ holds. Conversely, the equality in $(b)$ implies that 
        \[
            \Loco{t\otimes E_{\cat B} \mid \cat B \in \Supph(t)} = \Loco{t\otimes E_{\cat B} \mid \cat B \in \Supphnaive(t)},
        \]
    so the h-LGP follows from h-codetection by \cref{prop:hLGP}.
\end{proof}

\begin{Exa}\label{exa:DR-h-codetection}
    In light of \cref{exa:DR-tt-fields} and \cref{prop:naive-tt-fields}, we see that for the derived category $\Der(R)$ of a commutative ring $R$, the naive and genuine notions of homological support coincide. The previous corollary together with \cref{rem:hLGPannoying} then imply that $\Der(R)$ satisfies the  h-LGP if and only if it has the h-codetection property if and only if $R \in \Loco{\kappa(\mathfrak p)\mid \mathfrak p \in \Spec(R)}$. 
\end{Exa}

\begin{Rem}\label{rem:hlgphdetection}
	We remind the reader that the homological local-to-global principle implies the homological detection property (\cref{def:hdetection}). The following example shows that the converse does not hold.
\end{Rem}

\begin{Exa}\label{exa:absflat}
	Let $\Der(R)$ be the derived category of an absolutely flat ring $R$. It follows from \cite[Lemma 4.1]{Stevenson14} and \cref{prop:naive-tt-fields} that the h-detection property holds. However, \cite[Theorem 4.8]{Stevenson14} and \cite[Theorem~6.3]{Stevenson17} establish that the \mbox{h-LGP} holds if and only if $R$ is semi-artinian, in light of \cref{exa:DR-h-codetection}.
\end{Exa}

\section{Homological stratification}\label{sec:hstrat}

We now introduce  homological stratification and establish its basic properties. We begin with the following homological variant of \cite[Theorem~4.1]{bhs1}. 

\begin{Thm}\label{thm:hstratfundamental}
    Let $\cat T$ be a rigidly-compactly generated tt-category. The following conditions are equivalent:
    \begin{enumerate}
        \item the homological local-to-global principle holds for $\cat T$ and $\Loco{E_{\cat B}}$ is a minimal localizing ideal for all $\cat B \in \Spc^h(\cat T^c)$;
        \item for any $t \in \cat T$, we have $\Loco{t} = \Loco{E_{\cat B} \mid\cat B \in \Supph(t)}$;
        \item homological support induces a bijection
            \[ 
                \Supph\colon \big\{ \text{localizing ideals of $\cat T$} \big\} \xra{\cong} \big\{ \text{subsets of $\Spc^h(\cat T^c)$}\big\}.
            \]
    \end{enumerate}
\end{Thm}

\begin{proof}
    $(a) \Rightarrow (b)$: Consider $t \in \cat T$. The homological local-to-global principle implies
    \begin{align*}
        \Loco{t} & = \Loco{t \otimes E_{\cat B} \mid\cat B \in \Supph(t)} \\
        & \subseteq \Loco{E_{\cat B} \mid\cat B \in \Supph(t)}.
    \end{align*}
    On the other hand, minimality implies $\Loco{E_{\cat B}} = \Loco{t\otimes E_{\cat B}}$ for all $\cat B \in \Supph(t)$, which establishes the reverse inclusion.

    $(b) \Rightarrow (c)$: On the one hand, since $\Supph(E_{\cat B}) = \{\cat B\}$ for any $\cat B \in \Spc^h(\cat T^c)$, the map is always surjective. On the other hand, injectivity of the map is equivalent to the following statement:
    \[
        \forall t_1,t_2 \in \cat T\colon \Supph(t_1) = \Supph(t_2) \implies \Loco{t_1} = \Loco{t_2}.
    \]
    This is a direct consequence of $(b)$.

    $(c) \Rightarrow (a)$: We first check the minimality condition. To this end, let $\cat B$ be some homological prime and consider a nonzero $\cat L \subseteq \Loco{E_{\cat B}}$. It follows that $\emptyset \subseteq \Supph(\cat L) \subseteq \Supph(E_{\cat B})$. Injectivity of the map in $(c)$ guarantees that $\Supph(\cat L)$ is non-empty, for otherwise $\cat L = (0)$. Therefore, $\Supph(\cat L) = \{\cat B\}$, so applying $(c)$ again gives $\cat L = \Loco{E_{\cat B}}$. To prove the homological local-to-global principle, we compute
    \begin{align*}
        \Supph(\Loco{t \otimes E_{\cat B}\mid \cat B \in \Supph(t)}) & = \bigcup_{\cat B \in \Supph(t)}\Supph(t) \cap \Supph(E_{\cat B}) \\
        & = \Supph(t) \\
        & = \Supph(\Loco{t}).
    \end{align*}
    Using the injectivity in $(c)$ once more, we see that $\Loco{t \otimes E_{\cat B}\mid \cat B \in \Supph(t)} = \Loco{t}$, as desired.
\end{proof}

\begin{Def}\label{def:hstratification}
    A rigidly-compactly generated tt-category $\cat T$ is said to be \emph{homologically stratified} (or \emph{h-stratified}, for short) if it satisfies the equivalent conditions of \cref{thm:hstratfundamental}.
\end{Def}

\begin{Rem}\label{rem:hstratfundamental}
	It follows from \cref{lem:minimalityhsuppidentification} that $\Supph=\Supphnaive$ when  h-minimality  holds at all homological primes $\cat B$. Thus, in part $(a)$ of \cref{thm:hstratfundamental}, the homological local-to-global principle could be replaced with the h-codetection property; see \cref{prop:hLGP} and \cref{rem:hLGPannoying}.
\end{Rem}

\begin{Exa}\label{exa:ttfieldhstrat}
    Any tt-field $\cat F$ is h-stratified. Indeed, first note that $\cat F$ admits a unique homological prime by \cite[Theorem~5.17(a)]{BalmerKrauseStevenson19}. It then suffices to show that the only localizing ideals of $\cat F$ are the zero ideal and $\cat F$ itself, that is, $\unit\in\Loco{t}$ for any nonzero $t\in \cat F$. By the definition of tt-field, any nonzero object $t$ admits a nonzero rigid-compact object $c$ as a retract, hence $\Loco{c}\subseteq\Loco{t}$. Moreover, $\unit\in\Thickid{c}\subseteq\Loco{c}$ since $\Spc(\cat F^c)$ is a singleton by \cite[Proposition~5.15]{BalmerKrauseStevenson19}. Therefore, $\unit\in\Loco{t}$.
\end{Exa}

\begin{Prop}\label{prop:forced-minimal}
	Let $f^*\colon \cat T \to \cat F$ be a geometric functor to a tt-field and let $\cat B \in \Spc^h(\cat T^c)$ be the unique point in the image of $\varphi^h$. Assume that $\cat T$ has no nontrivial tensor-nilpotent objects. Then $\Loco{E_{\cat B}} = \Loco{f_*(\unit)}$ is a minimal localizing ideal.
\end{Prop}

\begin{proof}
	Balmer and Cameron \cite[Theorem 3.1]{BalmerCameron2021} establish that $E_{\cat B}$ is a direct summand of $f_*(\unit)$. Thus it suffices to prove that $\Loco{f_*(\unit)}$ is a minimal localizing ideal. Consider $0 \neq t \in \Loco{f_*(\unit)}$. If $f^*(t) = 0$ then $f_*(\unit) \otimes t = 0$ so that ${t \otimes t = 0}$. This contradicts $t \neq 0$ since we have assumed that there are no tensor-nilpotent objects in $\cat T$. Thus, $f^*(t) \neq 0$ and therefore $\unit \in \Loco{f^*(t)}$ by \cref{exa:ttfieldhstrat}.  Hence
	\[
        f_*(\unit) \in f_*\Loco{f^*(t)} \subseteq \Loco{t}
    \]
	by \cite[(13.4)]{BCHS1}. This proves that $\Loco{f_*(\unit)}$ is minimal.
\end{proof}

\begin{Cor}\label{cor:tt-fields-implies-h-minimality}
	If $\cat T$ has no nontrivial tensor-nilpotent objects and admits enough tt-fields, then it has h-minimality at all homological primes.
\end{Cor}

\begin{Exa}
    The corollary applies whenever $\cat T$ has the h-detection property and admits enough tt-fields.
\end{Exa}

\begin{Rem}
	Note that if $f^*\colon \cat T \to \cat F$ is a geometric functor to a tt-field  then $f^*(t)=0$ for any tensor-nilpotent object $t \in \cat T$. Hence, the argument in the proof of \cref{prop:forced-minimal} shows that if $t \in \Loco{f_*(\unit)}$ is tensor-nilpotent then $t^{\otimes 2} = 0$. In other words, this localizing ideal can only contain tensor-nilpotent objects of order at most 2.
\end{Rem}

\begin{Cor}\label{cor:ttfields1} 
    If $\cat T$ admits enough tt-fields then the following are equivalent:
	\begin{enumerate}
		\item $\cat T$ has the h-LGP property;
		\item $\cat T$ is h-stratified.
	\end{enumerate}
\end{Cor}

\begin{proof}
	Certainly $(b)$ implies $(a)$ by \cref{thm:hstratfundamental}. Conversely, h-LGP implies the \mbox{h-detection} property so, since we have enough tt-fields by assumption, we automatically get h-minimality by \cref{cor:tt-fields-implies-h-minimality}.
\end{proof}

\begin{Exa}\label{exa:DR-h-strat}
	The corollary applies to the derived category $\cat T=\Der(R)$ of any commutative ring since it admits enough tt-fields (\cref{exa:DR-tt-fields}). Therefore, $\Der(R)$ is h-stratified if and only if $R \in \Loco{\kappa(\mathfrak p)\mid \mathfrak p \in \Spec(R)}$, in light of \cref{exa:DR-h-codetection}. 
\end{Exa}

\begin{Exa}\label{exa:absflat-h-strat}
    In particular, for $R$ an absolutely flat ring, $\Der(R)$ is h-stratified if and only if $R$ is semi-artinian (\cref{exa:absflat}).
\end{Exa}

\begin{Par}
	We give two examples where homological minimality fails:
\end{Par}

\begin{Exa}
	Let $\cat T = \SH$ be the stable homotopy category. By \cite[Corollary~5.10]{Balmer20_nilpotence}, which relies on the nilpotence theorem of Devinatz--Hopkins--Smith \cite{DevinatzHopkinsSmith88}, the canonical map $\Spc^h(\SH^c)\to\Spc(\SH^c)$ is a bijection; cf.~\cref{def:nerves-of-steel} below. In particular, for a given prime number $p$, there is a homological prime~$\cat B_{p,\infty}$ corresponding to the tt-prime of all finite $p$-torsion spectra. The associated weak ring identifies with the mod $p$ Eilenberg--Mac Lane spectrum $H\Fp$, i.e., $E_{\cat B_{p,\infty}} = H\Fp$; see \cite[Corollary 3.6]{BalmerCameron2021}. However, it follows from Ravenel's work \cite{Ravenel84} that there are strict inclusions  
        \[
            0 \subsetneq \Loco{I} \subsetneq \Loco{H\Fp}
        \]
    where $I$ denotes the $p$-local Brown--Comenetz dual of the sphere spectrum. Indeed, on the one hand, we have $I \in \Loco{H\Fp}$ because the homotopy groups of $I$ are \mbox{$p$-torsion} and concentrated in non-negative degrees.  On the other hand, we have $H\Fp \not\in \Loco{I}$ because $H\Fp \otimes H\Fp \neq 0$ while $H\Fp \otimes I = 0$.
\end{Exa}

\begin{Exa}\label{exa:truncatedpolyring}
    Let $A$ be the truncated polynomial ring over a field $k$ considered by~\cite{Neeman00}. We grade $A$ as in \cite{DwyerPalmieri08}. The derived category $\Der(A)$ admits a unique tt-field $\Der(A) \to \Der(k)$, induced by the quotient $A\to k$ by the unique prime ideal of~$A$. The pure-injective object $E_{\cat B}$ associated to the unique  homological prime $\cat B$ is just the ordinary residue field $k$. Let $I=\Hom_k^*(A,k)$ be the graded $k$-dual of $A$, concentrated in non-positive degrees. We claim that $\Loco{I}$ is strictly contained in $\Loco{k}$. Indeed, we have $I\in \Loco{k}$ by \cite[Lemma~4.9]{DwyerPalmieri08}.  On the other hand, $k \not\in \Loco{I}$ since  $I\otimes k=0$ by \cite[Proposition~4.11]{DwyerPalmieri08} while $k\otimes k\neq0$.
\end{Exa}

\begin{Rem}\label{rem:h-LGP-fail}
    Recall from \cref{cor:ttfields1} that the homological local-to-global principle implies h-minimality in the presence of enough tt-fields. Hence the homological local-to-global principle fails in the previous two examples. Indeed, we have ${\Supph(I)=\emptyset}$ in both examples above.
\end{Rem}

\begin{Rem}
    The spectrum in \cref{exa:truncatedpolyring} is a single point, which shows that there is no obvious condition on the space $\Spc^h(\cat T^c)$ that ensures that the homological local-to-global principle holds. This should be contrasted with the classical  local-to-global principle, which holds whenever $\Spc(\cat T^c$) is noetherian \cite[Theorem~3.22]{bhs1}. 
\end{Rem}

\begin{Rem}
    Recall that the Bousfield class $A(t)$  of  an object $t \in \cat T$ is the localizing ideal $A(t) \coloneqq \ker(-\otimes t)$ and that the Bousfield lattice of $\cat T$ is the lattice of Bousfield classes ordered by reverse inclusion; see \cite[Section~8]{bhs1}.  Surprisingly, the following proposition shows that the classification of Bousfield classes is equivalent to the h-detection property whenever  the homological support agrees with the naive homological support. Indeed, it is immediate from the definition of naive homological support that  $\Supphnaive(t)\subseteq\Supphnaive(s)$ whenever $A(t)\le A(s)$.  Therefore, there is a well-defined order-preserving map
    \[
        \big\{\text{Bousfield classes of }\cat T\big\} \rightarrow \big\{\text{subsets of }\Spc^h(\cat T^c)\big\}
    \]
    sending $A(t)$ to $\Supphnaive(t)$.
\end{Rem}

\begin{Prop}\label{prop:bousfield-classes}
    The following are equivalent:
    \begin{enumerate}
        \item The h-detection property holds for $\cat T$.
        \item $\Supph(t)=\Supphnaive(t)$ for all $t \in \cat T$ and the Bousfield lattice of $\cat T$ is isomorphic to the lattice of subsets of $\Spc^h(\cat T^c)$ via the map sending $A(t)$ to $\Supph(t)$.
    \end{enumerate}
\end{Prop}

\begin{proof}
    $(a) \Rightarrow (b)$: The first statement follows from \cref{lem:supp-coincide-h-detection}. For the second, note that the well-defined order-preserving map $A(t)\mapsto\Supphnaive(t)=\Supph(t)$ is a surjection since every subset $S$ of $\Spc^h(\cat T^c)$ can be realized as the h-support of some object, say $\bigsqcup_{\cat B \in S}E_{\cat B}$, thanks to \cite[Proposition~4.3(b)]{Balmer20_bigsupport} and \cref{exa:suppEb}. We claim that
    \[
        A(t) \le A(s) \iff \Supph(t)\subseteq\Supph(s).
    \]
    Given the claim, the map sending $A(t)$ to $\Supph(t)$ is an order-preserving bijection with order-preserving inverse $\Supph(t)\mapsto A(t)$ and therefore an isomorphism of lattices. To prove the claim, it suffices to show that
    \[
        A(t) \le A(s) \impliedby \Supph(t)\subseteq\Supph(s).
    \]
    This follows immediately from the equality
    \[
        A(t)=\SET{s\in\cat T}{\Supph(s)\cap\Supph(t)=\emptyset},
    \]
    which holds by h-detection and the tensor product formula.
 
    $(b) \Rightarrow (a)$: If $\Supph(t)=\emptyset$ then $A(t)=A(0)$ and hence $t=0$.
\end{proof}

\begin{Rem}
    For any $\cat T$ satisfying the h-detection property, we obtain a description of the meet operation on its Bousfield lattice: $A(s)\wedge A(t)=A(s\otimes t)$; cf.~\cite[Theorem~8.8]{bhs1}.
\end{Rem}

\begin{Rem}
    If $\cat T$ is h-stratified then every localizing ideal is a Bousfield class and the lattice of localizing ideals of $\cat T$ is isomorphic to the Bousfield lattice of $\cat T$; cf.~\cite[Theorem 8.8]{bhs1}.
\end{Rem}

\begin{Exa}
    Let $R$ be an absolutely flat ring which is not semi-artinian. Then $\Der(R)$ is not h-stratified (\cref{exa:absflat-h-strat}), so the localizing subcategories of $\Der(R)$  are not classified by the subsets of $\Spec(R)$ via homological support. However, the h-detection property holds and hence the Bousfield lattice of $\Der(R)$ is isomorphic to the lattice of subsets of $\Spec(R)$.
\end{Exa}

\section{Homological stratification via cosupport}\label{sec:hstrat-cosupp}

We now establish  characterizations of homological stratification in terms of the behaviour of homological cosupport. The proof of the next proposition closely follows its tensor-triangular counterpart (\cite[Theorem 7.15]{BCHS1}):

\begin{Prop}\label{prop:hstratcosupp}
    Consider the following statements for a rigidly-compactly generated tt-category $\cat T$:
        \begin{enumerate}
            \item $\Loco{E_{\cat B}}$ is a minimal localizing ideal for all $\cat B \in \Spc^h(\cat T^c)$;
            \item $\Cosupph(\ihom{t_1,t_2}) = \Supph(t_1) \cap \Cosupph(t_2)$ for all $t_1,t_2 \in \cat T$;
            \item $\ihom{t_1,t_2} = 0$ implies $\Supph(t_1) \cap \Cosupph(t_2) = \emptyset$ for all $t_1,t_2 \in \cat T$.
        \end{enumerate}
    Then $(a) \Rightarrow (b) \Rightarrow (c)$. If $\cat T$ has the  h-detection property, then $(c) \Rightarrow (a)$, so all conditions are equivalent in this case. 
\end{Prop}

\begin{proof}
    $(a) \Rightarrow (b)$: The inclusion $\Cosupph(\ihom{t_1,t_2}) \subseteq \Supphnaive(t_1) \cap \Cosupph(t_2)$ holds unconditionally. Since $(a)$ guarantees  $\Supph(t_1) = \Supphnaive(t_1)$ by \Cref{lem:minimalityhsuppidentification}, we obtain $\Cosupph(\ihom{t_1,t_2}) \subseteq \Supph(t_1) \cap \Cosupph(t_2)$ and it remains to verify the reverse inclusion. Consider some $\cat B \in \Supph(t_1) \cap \Cosupph(t_2)$. In particular, $E_{\cat B} \otimes t_1 \neq 0$, hence $E_{\cat B} \in \Loco{E_{\cat B}\otimes t_1}$ since $\Loco{E_{\cat B}}$ is minimal by assumption. 
    Since $\cat B \in \Cosupph(t_2)$, we have 
        \[
            0 \neq \ihom{E_{\cat B},t_2} \in \Coloco{\ihom{E_{\cat B}\otimes t_1,t_2}}.
        \]
    By adjunction, we get $0 \neq\ihom{E_{\cat B}, \ihom{t_1,t_2}}$, so $\cat B \in \Cosupph(\ihom{t_1,t_2})$.

    $(b) \Rightarrow (c)$: This is a direct consequence of $\Cosupph(0) = \emptyset$.

    $(c) \Rightarrow (a)$: Consider two non-zero $t_1,t_2 \in \Loco{E_{\cat B}}$. We first claim that $\cat B \in \Cosupph(t_2)$. Indeed, if $\cat B \notin \Cosupph(t_2)$ then $\ihom{E_{\cat B},t_2}=0$, which would imply $\ihom{t_2,t_2}=0$, i.e., $t_2 = 0$, giving the desired contradiction. Moreover, $\cat B \in \Supph(t_1)$ by the h-detection property. It then follows that 
        \[
            \cat B \in \Supph(t_1) \cap \Cosupph(t_2),
        \]
    so $\ihom{t_1,t_2} \neq 0$ by $(c)$. The minimality criterion \cite[Lemma 7.12]{BCHS1} then shows that $\Loco{E_{\cat B}}$ is minimal. 
\end{proof}

\begin{Cor}\label{cor:hstratcosupp}
    A rigidly-compactly generated tt-category $\cat T$ is h-stratified if and only if the h-LGP holds for $\cat T$ and 
        \[
            \Cosupph(\ihom{t_1,t_2}) = \Supph(t_1) \cap \Cosupph(t_2)
        \]
    for all $t_1,t_2 \in \cat T$.
\end{Cor}

\begin{proof}
    The h-LGP implies the h-detection property (\cref{rem:hlgphdetection}), so the equivalence follows from \cref{thm:hstratfundamental} and \cref{prop:hstratcosupp}.
\end{proof}

\begin{Rem}
    \Cref{cor:hstratcosupp} is a homological analogue of the characterization of stratification in terms of cosupport \cite[Theorem 7.15]{BCHS1}. The goal of the rest of this section is to establish a variant of this result formulated purely in terms of support and cosupport. 
\end{Rem}

\begin{Rem}\label{rem:BCdual}
    The proof of the next theorem makes use of Brown--Comenetz duality. Recall from \cite[Definition 9.4]{BCHS1} that the Brown--Comenetz dual $I_c$ of a compact object $c \in \cat T$ is characterized by the formula
        \[
            \cat T(t,I_c) \cong \Hom_{\Z}(\cat T(c,t),\bbQ/\Z)
        \]
    for all $t \in \cat T$. Following the proof of \cite[Proposition 12.9]{BCHS1}, we compute
        \begin{equation}
            \Cosupph(I_c) = \Supphnaive(c) = \Supph(c)
        \end{equation}
    for any $c \in \cat T^c$, where the last equality comes from \cite[Proposition 4.4]{Balmer20_bigsupport}. 
\end{Rem}

\begin{Thm}\label{thm:hstratcosupp}
    For a rigidly-compactly generated tt-category $\cat T$, the following are equivalent:
    \begin{enumerate}
        \item $\cat T$ is homologically stratified;
        \item $\cat T$ satisfies the h-codetection property and, for all $t_1,t_2 \in \cat T$, we have 
            \begin{equation}\label{eq:cosuppformula}
                \Cosupph(\ihom{t_1,t_2}) = \Supph(t_1) \cap \Cosupph(t_2);
            \end{equation}
        \item $\ihom{t_1,t_2} = 0$ if and only if $\Supph(t_1) \cap \Cosupph(t_2) = \emptyset$ for all $t_1,t_2 \in \cat T$.
    \end{enumerate}
\end{Thm}

\begin{proof}
    Throughout this proof, we will refer to \eqref{eq:cosuppformula} as the \emph{cosupport formula}.
    
    $(a) \Rightarrow (b)$: By \cref{cor:hstratcosupp}, h-stratification implies the h-LGP for $\cat T$ as well as the cosupport formula. But the h-LGP implies the h-codetection property for $\cat T$ (see \cref{prop:hLGP} and \cref{rem:hLGPannoying}), so we get the statement of $(b)$.

    $(b) \Rightarrow (c)$: The \emph{only if} direction in $(c)$ follows directly from the cosupport formula. For the converse, suppose that $\Supph(t_1) \cap \Cosupph(t_2) = \emptyset$. The cosupport formula then shows that $\Cosupph(\ihom{t_1,t_2}) = \emptyset$, so h-codetection gives $\ihom{t_1,t_2} = 0$.

    $(c) \Rightarrow (a)$:  Taking $t_2 = I_{\unit}$ in $(c)$ and using \cref{rem:BCdual}, we have ${\ihom{t_1,I_{\unit}} = 0}$ if and only if $\Supph(t_1)=\emptyset$. Since $t_1 = 0$ if and only if $\ihom{t_1,I_{\unit}} = 0$, this establishes h-detection for $\cat T$.  We may thus appeal to the implication $(c) \Rightarrow (a)$ in \cref{prop:hstratcosupp} to deduce the minimality of $\Loco{E_{\cat B}}$ for all $\cat B \in \Spc^h(\cat T^c)$. 
    
    Specializing $(c)$ to $t_1 = \unit$ instead and using that $\Supph(\unit) = \Spc^h(\cat T^c)$, we get the h-codetection property for $\cat T$. \Cref{lem:minimalityhsuppidentification} together with \cref{cor:hLGP} then imply that $\cat T$ has the h-LGP. Finally, the characterization of h-stratification in \cref{thm:hstratfundamental} shows that $\cat T$ is h-stratified.
\end{proof}

\section{Comparison with the Balmer--Favi support}\label{sec:comparison}

In this section we will compare the homological support  with the Balmer--Favi support and, in particular, relate the h-detection and h-LGP properties with the usual detection and LGP properties. We assume some familiarity with \cite{bhs1}.

\begin{Not}
	We will write $\pi\colon \Spc^h(\cat T^c) \to \Spc(\cat T^c)$ for the surjective comparison map between the two spectra.
\end{Not}

\begin{Rem}\label{rem:comparisonforcompacts}
	For any compact object $x \in \cat T^c$, we have
    \[
        \Supph(x) = \Supphnaive(x) = \pi^{-1}(\supp(x))
    \]
	by \cite[Proposition 4.4]{Balmer20_bigsupport} where $\supp$ denotes the universal support theory for~$\cat T^c$.
\end{Rem}

\begin{Rem}
 A subset $W \subseteq \Spc(\cat T^c)$ is \emph{weakly visible} if it can be written as the intersection of a Thomason subset and the complement of a Thomason subset: $W= Y_1 \cap Y_2^c$.
 In this case, the object
$g_W \coloneqq e_{Y_1} \otimes f_{Y_2} \in \cat T$
only depends on~$W$ up to isomorphism.
A point $\cat P \in \Spc(\cat T^c)$ is said to be weakly visible if the singleton subset $\{\cat P\}$ is weakly visible,
and in this case we write $g_{\cat P} \coloneqq g_{\{\cat P\}}$.
Finally, the space $\Spc(\cat T^c)$ is \emph{weakly noetherian} if every point is weakly visible.
See \cite[Section~4]{BCHS1} for further details.
\end{Rem}

\begin{Lem}\label{lem:supph-gW}
	For any weakly visible subset $W \subseteq \Spc(\cat T^c)$, we have
	\[
        \Supph(g_W) = \Supphnaive(g_W) = \pi^{-1}(W).
    \]
\end{Lem}

\begin{proof}
	For any Thomason subset $Y\subseteq \Spc(\cat T^c)$, \cite[Lemmma~3.8]{bhs2} establishes that $\Supph(e_Y)=\pi^{-1}(Y)$ and $\Supph(f_Y)=\pi^{-1}(Y^c)$. Moreover, we have $\Supph(e_Y)=\Supphnaive(e_Y)$ and $\Supph(f_Y)=\Supphnaive(f_Y)$ since $e_Y$ is a (weak) coring and $f_Y$ is a (weak) ring (\Cref{prop:supp-coincide-weak-coring}). Finally, we claim that if $a,b \in \cat T$ both have the property that their naive and genuine h-supports coincide, then so does their tensor product $a \otimes b$. Indeed,
    \[
        \Supphnaive(a\otimes b) \subseteq \Supphnaive(a) \cap \Supphnaive(b) = \Supph(a) \cap \Supph(b) = \Supph(a\otimes b),
    \]
    while the reverse inclusion holds unconditionally. It follows that naive and genuine h-support coincide for $g_W$ for any weakly visible subset $W$. 
\end{proof}

\begin{Exa}
	For any weakly visible point $\cat P \in \Spc(\cat T^c)$, we have
	\[
        \Supph(g_{\cat P}) = \Supphnaive(g_{\cat P}) = \pi^{-1}(\{\cat P\}).
    \]
\end{Exa}

\begin{Lem}\label{lem:half-hom}
	We have an equality
	\[
        \Cosupph(\ihom{a,t}) = \Supph(a) \cap \Cosupph(t)
    \]
	whenever the object $a$ is compact or $a=g_W$ for a weakly visible subset $W \subseteq \Spc(\cat T^c)$.
\end{Lem}

\begin{proof}
	We always have the inclusion
	\[
        \Cosupph(\ihom{a,t})\subseteq \Supphnaive(a) \cap \Cosupph(t).
    \]
	For the objects $a$ under consideration, the naive and genuine h-support coincide by \cref{rem:comparisonforcompacts} and \cref{lem:supph-gW}, so we actually have
	\begin{equation}\label{eq:cosupphom} 
        \Cosupph(\ihom{a,t})\subseteq \Supph(a) \cap \Cosupph(t).
	\end{equation}
    Now suppose $a$ is compact and let $\cat B$ be in the right-hand side. We claim 
	\[
        \ihom{E_{\cat B},\ihom{a,t}} \neq 0.
    \]
	Otherwise, $\ihom{a,\ihom{E_{\cat B},t}} = 0$ so that $\ihom{e_{\supp(a)},\ihom{E_{\cat B},t}} = 0$. Hence $\ihom{E_{\cat B},t} \simeq \ihom{f_{\supp(a)},\ihom{E_{\cat B},t}}$. But $\ihom{E_{\cat B},t} \neq 0$ by hypothesis. It follows that $E_{\cat B}\otimes f_{\supp(a)}\neq 0$. That is, $\cat B \in \Supphnaive(f_{\supp(a)}) = \pi^{-1}(\supp(a)^c) = \pi^{-1}(\supp(a))^c = \Supph(a)^c$ by \cref{lem:supph-gW}, which contradicts the hypothesis. This proves the claim when $a$ is compact.

	Now suppose $a=e_Y$. Again, consider $\cat B$ in the right-hand side of \eqref{eq:cosupphom}. If it is not contained in the left-hand side then $\ihom{e_Y,\ihom{E_{\cat B},t}} =0$ so that $\ihom{x,\ihom{E_{\cat B},t}} = 0$ for all compact $x$ with $\supp(x) \subseteq Y$. That is, $\cat B \not\in \Cosupph(\ihom{x,t})$ so that $\cat B \not\in \Supph(x)$. Hence 
	\[
        \cat B \not\in \bigcup_{x \in \cat T^c_Y} \Supph(x) = \pi^{-1}(\bigcup_{x \in \cat T^c_Y} \supp(x)) = \pi^{-1}(Y) = \Supph(e_Y)
    \]
	which contradicts the hypothesis.

	Now suppose $a=f_Y$. Consider any $\cat B$ contained in the right-hand side of \eqref{eq:cosupphom}. Note that $\cat B \not\in \Supph(e_Y)$ since the latter is the complement of $\Supph(f_Y)$. Hence $\cat B \not\in \Cosupph(\ihom{e_Y,t})$. That is, $\ihom{E_{\cat B},\ihom{e_Y,t}} = 0$. Therefore, 
    \[
        0 \neq \ihom{E_{\cat B},t} \simeq \ihom{f_Y,\ihom{E_B,t}}.
    \]
    Hence $\cat B \in \Cosupph(\ihom{f_Y,t})$.

    It follows that we have the equality for a tensor product like $g_W = e_{Y_1} \otimes f_{Y_2}$.
\end{proof}

\begin{Prop}\label{prop:support-relations}
	Assume that $\Spc(\cat T^c)$ is weakly noetherian. Then 
	\begin{enumerate}
		\item $\pi(\Supph(t)) \subseteq \Supp(t)$ for all $t \in \cat T$ with equality when h-detection holds.
		\item $\pi(\Cosupph(t)) \subseteq \Cosupp(t)$ for all $t \in \cat T$ with equality when h-codetection holds.
	\end{enumerate}
\end{Prop}

\begin{proof}
	We proved in \cite[Proposition 3.10]{bhs2} that the inclusion $\pi(\Supph(t)) \subseteq \Supp(t)$ always holds. On the other hand, consider  the purported inclusion 
    	\[
            \pi(\Cosupph(t)) \subseteq \Cosupp(t).
        \]
	If $\cat P$ is not contained in the right-hand side, that is,  $\ihom{g_{\cat P},t} = 0$, then
    	\[
            \emptyset = \Cosupph(\ihom{g_{\cat P},t}) = \Supph(g_{\cat P})\cap \Cosupph(t)
    	= \pi^{-1}(\{ \cat P\}) \cap \Cosupph(t)
        \]
	by \cref{lem:half-hom} and \cref{lem:supph-gW}. Thus, $\cat P \not\in \pi(\Cosupph(t))$.

	Now for the equality in $(a)$: If $\cat P \in \Supp(t)$ then $g_{\cat P} \otimes t \neq 0$. Hence, the \mbox{h-detection} property implies
    	\[
            \emptyset \neq \Supph(g_{\cat P} \otimes t) = \Supph(g_{\cat P}) \cap \Supph(t) = \pi^{-1}(\{\cat P\})\cap\Supph(t)
        \]
	so that $\cat P \in \pi(\Supph(t))$.

	For the equality in $(b)$: If $\cat P \in \Cosupp(t)$ then $\ihom{g_{\cat P},t} \neq 0$. Hence, the h-codetection property implies
		\[
            \emptyset \neq \Cosupph(\ihom{g_{\cat P},t}) = \Supph(g_{\cat P})\cap \Cosupph(t) = \pi^{-1}(\{\cat P\}) \cap \Cosupph(t)
        \]
	by \cref{lem:half-hom} and \cref{lem:supph-gW} so that $\cat P \in \pi(\Cosupph(t))$.
\end{proof}

\begin{Cor}\label{cor:h-detect-implies-detect}
	Assume that $\Spc(\cat T^c)$ is weakly noetherian. Then
	\begin{enumerate}
		\item h-detection implies detection.
		\item h-codetection implies codetection.
	\end{enumerate}
\end{Cor}

\begin{proof}
	This immediately follows from \cref{prop:support-relations}.
\end{proof}

\begin{Rem}
	The detection property does not in general imply the h-detection property and the codetection property does not in general imply the h-codetection property. Indeed, as explained in \cite[Example 5.5]{bhs2},  Neeman \cite{Neeman00}  provides an example of a non-noetherian commutative ring $R$ whose spectrum $\Spec(R)$ is a single point  with the property that $\Der(R)$ contains a nontrivial tensor-nilpotent object $I\in \Der(R)$; cf.~\cref{exa:truncatedpolyring}. The tensor product property forces ${\Supph(I) =\emptyset}$. Thus $\Der(R)$ does not have the h-detection property and hence it does not have the h-LGP property or h-codection property (\cref{exa:DR-h-codetection}), either. However, by \cite[Theorem 3.22]{bhs1}  it does have the LGP (which is equivalent to the codetection property by \cite[Theorem 6.4]{BCHS1}) since the spectrum $\Spc(\Der(R)^c)=\Spec(R)=*$ is a noetherian space.
\end{Rem}

\begin{Par}
    For weak rings, it is possible to prove an unconditional comparison result between the homological and Balmer--Favi support: 
\end{Par}

\begin{Prop}\label{prop:supphweakrings}
    Assume that $\Spc(\cat T^c)$ is weakly noetherian. For any weak ring $w \in \cat T$, we have $\pi(\Supph(w)) = \Supp(w)$.
\end{Prop}

\begin{proof}
	The $\subseteq$ inclusion always holds (\cref{prop:support-relations}), so it suffices to establish $\pi(\Supph(w)) \supseteq\Supp(w)$. Let $\cat P \in \Supp(w)$. We have $\{\cat P\} = \supp(x) \cap \gen(\cat P)$ for some $x \in \cat T^c$  by \cite[Remark 2.8]{bhs1}. Moreover, replacing $x$ by  $x \otimes x^{\vee}$ we may assume that $x$ is a (weak) ring. Then $w \otimes g_{\cat P} \neq 0$ implies that $w \otimes x \otimes f_{\gen(\cat P)^c} \neq 0$ is a nonzero weak ring. The homological support has the detection property for weak rings by \cite[Theorem 4.7]{Balmer20_bigsupport}. Hence
	\begin{align*}
    	\emptyset \neq \Supph(w \otimes x \otimes f_{\gen(\cat P)^c}) &= \Supph(w)\cap \Supph(x) \cap \Supph(f_{\gen(\cat P)^c}) \\
     &= \Supph(w) \cap \pi^{-1}(\{\cat P\}) 
    \end{align*}
    by the tensor product property~\eqref{eq:tensor-product}, \cref{rem:comparisonforcompacts} and \cref{lem:supph-gW}. Therefore $\cat P \in \pi(\Supph(w))$, as desired.
\end{proof}

\section{Base change for homological support}\label{sec:base-change}

We now turn to studying the behaviour of the homological (co)support in a relative situation. Here we discover some surprises, as the behavior of the homological support deviates strongly from that of the Balmer--Favi support.

\begin{Par}
    The following result, due to Balmer, is crucial:
\end{Par}

\begin{Lem}\label{lem:direct-summand}
    Let $f^*:\cat T \to \cat S$ be a geometric functor.
	For any homological prime $\cat B \in \Spc^h(\cat S^c)$,	 $E_{\varphi^h(\cat B)} \in \cat T$ is a direct summand  of ${f_*(E_{\cat B})\in \cat T}$.
\end{Lem}

\begin{proof}
	This is established in \cite[Lemma 5.6]{Balmer20_bigsupport}.
\end{proof}

\begin{Prop}\label{prop:base-inclusion}
	Let $f^*\colon\cat T \to \cat S$ be a geometric functor. For any $t \in \cat T$, we have
    \[
    (\varphi^h)^{-1}(\Supph(t)) \subseteq \Supph(f^*(t)) \subseteq (\varphi^h)^{-1}(\Supphnaive(t)) \subseteq \Supphnaive(f^*(t)).
    \]
\end{Prop}

\begin{proof}
    Let $\cat B \in \Spc^h(\cat S^c)$ be such that $\varphi^h(\cat B) \in \Supph(t)$. By definition, this means $\ihom{t,E_{\varphi^h(\cat B)}} \neq 0$. Hence $\ihom{t,f_*E_{\cat B}} \neq 0$ by \cref{lem:direct-summand}. That is, $f_*\ihom{f^*(t),E_{\cat B}} \neq 0$. Therefore, $\ihom{f^*(t),E_{\cat B}} \neq 0$, i.e., $\cat B \in \Supph(f^*(t))$. This establishes the first inequality.

    Now suppose $\cat B \in \Supph(f^*(t))$. Then by the first inequality applied to $E_{\varphi^h(\cat B)}$, we have
    \begin{equation}\label{eq:suppbasechangeEb}
        \cat B \in (\varphi^h)^{-1}(\{\varphi^h(\cat B)\}) \subseteq \Supph(f^*(E_{\varphi^h(\cat B)}))
    \end{equation}
    and thus $\cat B \in \Supph(f^*(t)) \cap \Supph(f^*(E_{\varphi^h(\cat B)}))=\Supph(f^*(t \otimes E_{\varphi^h(\cat B)}))$. Therefore, $f^*(t \otimes E_{\varphi^h(\cat B)}) \neq 0$. Hence, $t \otimes E_{\varphi^h(\cat B)} \neq 0$. That is, $\varphi^h(\cat B) \in \Supphnaive(t)$. This establishes the second inequality.

    Finally, suppose $\cat B \in (\varphi^h)^{-1}(\Supphnaive(t))$. That is, $E_{\varphi^h(\cat B)}\otimes t \neq 0$. \Cref{lem:direct-summand} then implies $f_*(E_{\cat B}) \otimes t \neq 0$. By the projection formula, this means $f_*(E_{\cat B} \otimes f^*(t)) \neq 0$ so $E_{\cat B} \otimes f^*(t) \neq 0$. Hence $\cat B \in \Supphnaive(f^*(t))$. This establishes the third inequality.
\end{proof}

\begin{Cor}\label{cor:AV-scott}
    Let $f^*\colon\cat T \to \cat S$ be a geometric functor. The Avrunin--Scott identity
        \begin{equation*}
            (\varphi^h)^{-1}(\Supph(t)) = \Supph(f^*(t))
        \end{equation*}
    holds 
        \begin{enumerate}
            \item for all weak rings $t \in \cat T$; and
            \item for all weak corings $t \in \cat T$; and
            \item for all objects $t \in \cat T$ if~$\cat T$ has the h-detection property.
        \end{enumerate}
\end{Cor}

\begin{proof}
    In cases $(a)$, $(b)$ and $(c)$, we have $\Supph(t)=\Supphnaive(t)$ by \cref{prop:supp-coincide-weak-coring} and \cref{lem:supp-coincide-h-detection}, hence the result follows from  \cref{prop:base-inclusion}.
\end{proof} 

\begin{Exa}\label{exa:AV-scott-EB}
	We always have $(\varphi^h)^{-1}(\{\cat B\}) = \Supph(f^*(E_{\cat B}))$.
\end{Exa}

\begin{Rem}
	The above corollary (\cref{cor:AV-scott})
    shows
    that the Avrunin--Scott identity holds unconditionally whenever $\cat T$ has the h-detection property. This contrasts significantly with the situation for the Balmer--Favi support; cf.~\cite[Corollary~14.19]{BCHS1} and \cite[Corollary~12.8]{BCHNPS_descent}.
\end{Rem}

\begin{Cor}\label{cor:surjectivitycriterion}
	Let $f^*\colon \cat T \to \cat S$ be a geometric functor. If $\cat T$ has the h-detection property then the following are equivalent:
	\begin{enumerate}
		\item $f^*$ is conservative;
		\item $f^*$ is nil-conservative, \ie $f^*$ is conservative  on weak rings;
		\item $\varphi^h$ is surjective.
	\end{enumerate}
\end{Cor}

\begin{proof}
	The equivalence of $(b)$ and $(c)$ is \cite[Theorem 1.8]{BarthelCastellanaHeardSanders22b} and $(a)$ certainly implies $(b)$. The implication $(c) \Rightarrow (a)$ follows from \cref{cor:AV-scott}.
\end{proof}

\begin{Prop}\label{prop:AV-S-h-strat}
	Let $f^*\colon\cat T \to \cat S$ be a geometric functor. Suppose $\cat S$  satisfies \mbox{h-minimality} for all homological primes. Then
        \[ 
            (\varphi^h)^{-1}(\Supph(t)) = \Supph(f^*(t))
        \]
	for all $t \in \cat T$.
\end{Prop}

\begin{proof}
    Let $\cat B \in \Spc^h(\cat S^c)$ be such that $\varphi^h(\cat B) \in \Supph(t)$, i.e., $0 \neq \ihom{t,E_{\varphi^h(\cat B)}}$. By adjunction and using \cref{lem:direct-summand}, this implies
        \[
            0 \neq \ihom{t,f_*E_{\cat B}} \simeq  f_*\ihom{f^*t, E_{\cat B}}.
        \]
    In particular, we see that $0\neq \ihom{f^*t, E_{\cat B}}$, that is $\cat B \in \Supph(f^*t)$. 

    For the reverse inclusion, take $\cat B \in \Supph(f^*t)$. We first claim that $\cat B \in \Cosupph(f^!E_{\varphi^h(\cat B)})$. Indeed, since $\ihom{E_{\varphi^h(\cat B)},E_{\varphi^h(\cat B)}}\neq 0$, \cref{lem:direct-summand} gives 
        \[
            0 \neq \ihom{f_*E_{\cat B},E_{\varphi^h(\cat B)}} \simeq f_*\ihom{E_{\cat B},f^!E_{\varphi^h(\cat B)}},
        \]
    so in particular $\cat B \in \Cosupph(f^!E_{\varphi^h(\cat B)})$. By assumption, $\cat S$ satisfies \mbox{h-minimality} at $\cat B$, so \cref{prop:hstratcosupp} implies 
        \[
            \cat B \in \Supph(f^*t) \cap \Cosupph(f^!E_{\varphi^h(\cat B)}) =\Cosupph(\ihom{f^*t,f^!E_{\varphi^h(\cat B)}}). 
        \]
    Hence $0 \neq \ihom{f^*t,f^!E_{\varphi^h(\cat B)}} \simeq f^!\ihom{t,E_{\varphi^h(\cat B)}}$, so $\varphi^h(\cat B) \in \Supph(t)$.
\end{proof}

\begin{Cor}
	Let $f^*\colon\cat T \to \cat S$ be a conservative geometric functor. Suppose $\cat S$ is h-stratified. Then $\cat T$ has the h-detection property.
\end{Cor}

\begin{proof}
	If $\Supph(t)=\emptyset$ then $\Supph(f^*(t))=\emptyset$ by \cref{prop:AV-S-h-strat} and hence $f^*(t)=0$ since \mbox{h-stratified} categories have the h-detection property. Since $f^*$ is conservative, we conclude that $t=0$.
\end{proof}

\begin{Par}
    While this paper does not systematically develop the theory of homological cosupport, we do include the following cosupport-version of \cref{prop:AV-S-h-strat}, which  will be useful  in \cref{sec:weakly-descendable} below.
\end{Par}

\begin{Prop}\label{prop:AV-coS-h-strat}
    Let $f^*\colon\cat T \to \cat S$ be a geometric functor. Suppose $\cat S$ satisfies \mbox{h-minimality} for all homological primes. Then 
	\[ 
            (\varphi^h)^{-1}(\Cosupph(t)) = \Cosupph(f^!(t))
    \]
    for all $t \in \cat T$.
\end{Prop}

\begin{proof}
    The proof is similar to the one of \cref{prop:AV-S-h-strat}. Let $\cat B \in \Spc^h(\cat S^c)$ be such that $\varphi^h(\cat B) \in \Cosupph(t)$, i.e., $0 \neq \ihom{E_{\varphi^h(\cat B)},t}$. By adjunction and using \cref{lem:direct-summand}, this implies
        \[
            0 \neq \ihom{f_*E_{\cat B},t} \simeq  f_*\ihom{E_{\cat B},f^!t}.
        \]
    In particular, we see that $0\neq \ihom{E_{\cat B},f^!t}$, that is $\cat B \in \Cosupph(f^!t)$. 

    To prove the reverse inclusion, consider some $\cat B \in \Cosupph(f^!t)$. Since $\cat S$ satisfies \mbox{h-minimality} at $\cat B$, \eqref{eq:suppbasechangeEb} and \cref{prop:hstratcosupp} give 
        \[
            \cat B \in \Supph(f^*E_{\varphi^h(\cat B)}) \cap \Cosupph(f^!t) =\Cosupph(\ihom{f^*E_{\varphi^h(\cat B)},f^!t}). 
        \]
    Hence $0 \neq \ihom{f^*E_{\varphi^h(\cat B)},f^!t} \simeq f^!\ihom{E_{\varphi^h(\cat B)}, t}$, so $\varphi^h(\cat B) \in \Cosupph(t)$.
\end{proof}

\begin{Prop}
	Let $f^*\colon\cat T \to \cat S$ be a geometric functor. Then
    \begin{enumerate}
        \item $\varphi^h(\Supph(R)) = \Supph(f_*(R))$ for any weak ring $R \in \cat S$.
        \item If $\cat S$ has the h-detection property then $\Supph(f_*(s)) \subseteq \varphi^h(\Supph(s))$ for all $s\in \cat S$.
    \end{enumerate}
\end{Prop}

\begin{proof}
    For the first statement, 
    observe that
    \begin{flalign*}
    \cat B \in \varphi^h(\Supph(R)) &\Longleftrightarrow (\varphi^h)^{-1}(\{\cat B\}) \cap \Supph(R) \neq \emptyset &&\\
      &\Longleftrightarrow \Supph(f^*(E_{\cat B})\otimes R) \neq \emptyset && \eqname{{\cref{cor:AV-scott}}}\\
      &\Longleftrightarrow f^*(E_{\cat B}) \otimes R \neq 0 && \eqname{{\cref{rem:hsuppdetectsweakrings}}} \\
     &\Longleftrightarrow f_*(f^*(E_{\cat B}) \otimes R)\neq 0 && \eqname{\cite[Remark~13.12]{BCHS1}}\\
     &\Longleftrightarrow E_{\cat B}\otimes f_*(R) \neq 0 && \\
     &\Longleftrightarrow \cat B \in \Supph(f_*(R)). &&
    \end{flalign*}
	For the second statement, suppose $\cat B \in \Supph(f_*(s))$. This implies $E_{\cat B} \otimes f_*(s) \neq 0$. That is, $f_*(f^*(E_{\cat B})\otimes t)\neq 0$. Hence $f^*(E_{\cat B}) \otimes s\neq 0$. The h-detection property for~$\cat S$ then implies $\Supph(f^*(E_{\cat B})\otimes s) \neq \emptyset$. It then follows that $(\varphi^h)^{-1}(\{\cat B\}) \cap \Supph(s) \neq \emptyset$ by using \cref{exa:AV-scott-EB}. Hence $\cat B \in \varphi^h(\Supph(s))$.
\end{proof}

\begin{Prop}\label{prop:image-of-finite}
	Let $f^*\colon\cat T \to \cat S$ be a finite localization. The induced map on homological spectra $\varphi^h\colon \Spc^h(\cat S^c) \to \Spc^h(\cat T^c)$ is an embedding with image $\im(\varphi^h) = \pi^{-1}(\im(\varphi))$.
\end{Prop}

\begin{proof}
	First we establish injectivity. Let $\cat C_1, \cat C_2 \in \Spc^h(\cat S^c)$ and suppose $\varphi^h(\cat C_1) = \varphi^h(\cat C_2)$. Then $E_{\varphi^h(\cat C_1)} \otimes E_{\varphi^h(\cat C_2)} \neq 0$. Hence $f_*(E_{\cat C_1}) \otimes f_*(E_{\cat C_2}) \neq 0$ by \cref{lem:direct-summand}. Moreover, since $f^*$ is a finite localization, it is smashing, so $f_*$ preserves the tensor product (but not necessarily the unit).  Therefore, $f_*(E_{\cat C_1}) \otimes f_*(E_{\cat C_2}) \simeq f_*(E_{\cat C_1} \otimes E_{\cat C_2})$ and hence we conclude that $E_{\cat C_1} \otimes E_{\cat C_2} \neq 0$. Thus $\cat C_1 = \cat C_2$.

    Recall from \cite[Remark 3.4]{Balmer20_nilpotence} and \cite[Example 2.8]{bhs2} that a basis of closed sets for the topology on $\Spc^h(\cat S^c)$  is given by the $\Supph(x)$ for $x \in \cat S^c$. Since~$f^*$ is a finite localization, $x\oplus \Sigma x = f^*(c)$ for some compact $c \in \cat T^c$ and (replacing $c$ with $c \otimes c^\vee$) we can assume $c$ is a (weak) ring. Then by \cref{cor:AV-scott} we have
    \[
        \varphi^h(\Supph(x))=\varphi^h(\Supph(f^*(c))) = \varphi^h((\varphi^h)^{-1}(\Supph(c))) = \Supph(c) \cap \im(\varphi^h).
    \]
    Moreover since $\varphi^h$ is injective, it preserves intersections. It follows that for an arbitrary closed subset $Z \subseteq \Spc^h(\cat S^c)$, we similarly have $\varphi^h(Z) = Z' \cap \im(\varphi^h)$ for some closed subset $Z' \subseteq \Spc^h(\cat T^c)$. This establishes that $\varphi^h$ is an embedding.

    To describe the image of $\varphi^h$, note that the object $f_*(\unitS)$ is the right idempotent of the finite localization so $\Supph(f_*(\unitS)) = \pi^{-1}(\im(\varphi))$ by \cref{lem:supph-gW}  and $\Supph(f_*(\unitS))=\im(\varphi^h)$ by \cite[Theorem 5.12]{Balmer20_bigsupport}.
\end{proof}

\begin{Cor}\label{cor:cartesian}
	Let $f^*\colon\cat T \to \cat S$ be a finite localization. Then the induced diagram
	\[
    \begin{tikzcd}
		\Spc^h(\cat S^c) \ar[d,"\pi"'] \ar[hook,r,"\varphi^h"] &\Spc^h(\cat T^c) \ar[d,"\pi"] \\
		\Spc(\cat S^c) \ar[hook,r,"\varphi"] &\Spc(\cat T^c)
	\end{tikzcd}
    \]
	is cartesian.
\end{Cor}

\begin{proof}
	Let $\cat B \in \Spc^h(\cat T^c)$. We claim that $\cat B$ is in the image of $\varphi^h$ if and only if $\pi(\cat B)$ is  in the image of $\varphi$, but this is precisely what $\im(\varphi^h)=\pi^{-1}(\im(\varphi))$ says.
\end{proof}

\begin{Cor}\label{cor:cartesian2}
    Let $f^*\colon \cat T\to \cat S$ be a geometric functor and let $V \subseteq \Spc(\cat T^c)$ be the complement of a Thomason subset. The induced diagram
    \[
    \begin{tikzcd}
        \Spc^h(\cat S^c) \ar[r,"\varphi^h"] & \Spc^h(\cat T^c) \\
        \Spc^h(\cat S(\varphi^{-1}(V))^c) \ar[u,hook] \ar[r] & \Spc^h(\cat T(V)^c) \ar[u,hook]
    \end{tikzcd}
    \]
    is cartesian.
\end{Cor}

\begin{proof}
    This follows from  \cref{prop:image-of-finite} and the definitions.
\end{proof}

\section{Weakly descendable functors and descent}\label{sec:weakly-descendable}

\begin{Ter}
    Let $(f_i^*\colon \cat T \to \cat S_i)_{i \in I}$ be a family of geometric functors between rigidly-compactly generated tt-categories, indexed on a (not necessarily finite) set~$I$. We refer to $(f_i^*)_{i \in I}$ as a \emph{geometric family}.
\end{Ter}

\begin{Def}\label{def:weaklydescendable}
    The geometric family $(f_i^*)_{i \in I}$ is said to be \emph{weakly descendable} if 
    \begin{equation}\label{eq:weakdescentcondition}
        \unit \in \Loco{(f_i)_*\unit \mid i \in I}
    \end{equation}
    in $\cat T$.
\end{Def}

\begin{Rem}\label{rem:descendable}
	The terminology of \cref{def:weaklydescendable} is inspired by the notion of a descendable algebra, introduced by \cite{Mathew2016Galois}. More generally, we say that a geometric family $(f_i^*)_{i \in I}$ is \emph{descendable} if 
    \begin{equation}\label{eq:descentcondition}
        \unit \in \Thickid{(f_i)_*\unit \mid i \in I}
    \end{equation}
    in $\cat T$. The situation studied in loc.~cit.~is the special case when $I$ is a singleton and~$f^*$ is given by base-change along a commutative algebra.
\end{Rem}

\begin{Exa}\label{ex:strictweaklydescendable}
    Base-change along $\bbZ_{(p)} \to \bbQ \times \bbZ/p$ provides a geometric functor on derived categories which is weakly descendable but not descendable. 
\end{Exa}

\begin{Lem}\label{lem:weakdescentcharacterization}
    For a geometric family $(f_i^*)_{i \in I}$, the following 
    are equivalent:
    \begin{enumerate}
        \item $(f_i^*)_{i \in I}$ is weakly descendable;
        \item $(f_i^!)_{i \in I}$ is jointly conservative;
        \item for any $t_1,t_2 \in \cat T$, we have
            \[
                t_1 \in \Loco{t_2} \iff \forall i \in I\colon f_i^*(t_1) \in \Loco{f_i^*(t_2)}.
            \]
    \end{enumerate}
\end{Lem}

\begin{proof}
    We first prove $(a) \Leftrightarrow (b)$. Write $\cat L = \Loco{(f_i)_*\unit \mid i \in I}$ and consider its right orthogonal $\cat L^{\perp}$ in $\cat T$. Note that 
    \begin{align*}
        t \in \cat L^{\perp} & \iff \ihomT{(f_i)_*\unit,t} = 0 \text{ for all } i \in I \\ 
        & \iff (f_i)_*f_i^!t = 0  \text{ for all } i \in I \\
        & \iff f_i^!t = 0  \text{ for all } i \in I.
    \end{align*}
	The first equivalence follows from \cite[Definition 2.9]{BCHS1}, the second equivalence follows from the internal adjunction  \cite[(2.18)]{BalmerDellAmbrogioSanders16}, and the final equivalence follows from the observation that a left adjoint is conservative on the essential image of its right adjoint, by the unit-counit identity.

    $(a) \Rightarrow (c)$: If $t_1 \in \Loco{t_2}$ then $f_i^*(t_1) \in \Loco{f_i^*(t_2)}$ for all $i \in I$, since $f_i^*$ are tt-functors, so it remains to prove the reverse implication in $(c)$. Tensoring \eqref{eq:weakdescentcondition} with $t_1$ and using the projection formula gives
        \[
            t_1 \in \Loco{t_1\otimes (f_i)_*f_i^*(\unit)\mid i \in I} = \Loco{(f_i)_*f_i^*(t_1)\mid i \in I}.
        \]
    Assuming $f_i^*(t_1) \in \Loco{f_i^*(t_2)}$ and using the projection formula again, we deduce $(f_i)_*f_i^*(t_1) \in \Loco{(f_i)_*f_i^*(t_2)}$. Now varying over $i \in I$ yields 
        \begin{align*}
            t_1 \in \Loco{(f_i)_*f_i^*(t_1)\mid i \in I} 
                & \subseteq \Loco{(f_i)_*f_i^*(t_2)\mid i \in I} \\
                & = \Loco{t_2 \otimes (f_i)_*f_i^*(\unit)\mid i \in I} \\
                & \subseteq \Loco{t_2},
        \end{align*}
    as claimed. 

    $(c) \Rightarrow (a)$: By the triangle identity for the 
    $f_i^* \dashv (f_i)_*$
    adjunction, $f_i^*(\unit)$ is a retract of $f_i^*(f_i)_*f_i^*(\unit) \simeq f_i^*(f_i)_*(\unit)$, hence $f_i^*(\unit) \in \Loco{f_i^*(f_i)_*(\unit))}$ for all $i \in I$. Therefore, condition $(c)$ implies $\unit \in \Loco{(f_i)_*(\unit)\mid i \in I}$, so $(f_i^*)_{i \in I}$ is weakly descendable.
\end{proof}

\begin{Exa}
    A geometric functor $f^*\colon \cat T \to \cat S$ is weakly descendable if and only if its double-right adjoint $f^!\colon \cat T\to \cat S$ is conservative.
\end{Exa}

\begin{Rem}
    Condition $(c)$ of \cref{lem:weakdescentcharacterization} has been used in \cite{BBIKP_stratification} to establish a fiberwise criterion for stratification. Its tt-geometric incarnation is studied further in work of Juan Omar G\'omez \cite{Gomez_fibrewise2025}.
\end{Rem}

\begin{Def}\label{def:weaklystronglyclosed}
    A geometric family $(f_i^*)_{i \in I}$ is said to be \emph{weakly closed} if ${(f_i)_*(\unit)\in \cat T}$ is compact for all $i\in I$. It is said to be \emph{strongly closed} if $(f_i)_*:\cat S_i \to \cat T$ preserves compact objects for all $i\in I$.
\end{Def}

\begin{Prop}\label{prop:surjectivity}
    Let $(f_i^*\colon \cat T \to \cat S_i)_{i\in I}$ be a geometric family. Consider the following statements:
    \begin{enumerate}
        \item $(f_i^*)_{i\in I}$ is weakly descendable;
        \item $(f_i^*)_{i \in I}$ is jointly conservative;
        \item $(f_i^*)_{i \in I}$ is jointly nil-conservative;
        \item the induced map on homological spectra \[
        \varphi^h = \sqcup \varphi_i^h\colon \bigsqcup_{i \in I}\Spc^h(\cat S_i^c) \to \Spc^h(\cat T^c) \] is surjective;
        \item the induced map on tt-spectra \[
        \varphi = \sqcup \varphi_i\colon \bigsqcup_{i \in I}\Spc(\cat S_i^c) \to \Spc(\cat T^c) \] is surjective.
    \end{enumerate}
    Then $(a) \Rightarrow (b) \Rightarrow (c) \Leftrightarrow (d) \Rightarrow (e)$. Moreover, if $(f_i^*)_{i \in I}$ is weakly closed then $(e) \Rightarrow (a)$ and hence all the statements are equivalent.
\end{Prop}

\begin{proof}
    $(a)\Rightarrow(b)$: This follows from \cref{lem:weakdescentcharacterization} and the proof of \cite[Proposition 13.21]{BCHS1}.
    $(b)\Rightarrow(c)$: This is immediate from the definitions.
    $(c)\Leftrightarrow(d)$: This is \cite[Theorem 1.8]{BarthelCastellanaHeardSanders22b}. $(d) \Rightarrow (e)$: This holds because $\pi$ is surjective.
    $(e)\Rightarrow(a)$: Since $\varphi$ is surjective, we have
    \[
        \Spc(\cat T^c)=\im\varphi=\bigcup_{i\in I}\im\varphi_i.
    \]
    If the $f_i^*$ are weakly closed then, by definition, 
    each $(f_i)_*(\unit)$ is compact. 
    We then have $\im\varphi_i \subseteq \supp((f_i)_*(\unit))$ by the argument given at the beginning of the proof of \cite[Theorem~1.7]{Balmer18}.
    Hence
    \[
        \Spc(\cat T^c)=\bigcup_{i\in I}\im\varphi_i\subseteq\bigcup_{i \in I}\supp((f_i)_*(\unit))
    \]
    so that 
$\cat T^c=\Thickid{(f_i)_*\unit \mid i \in I}$ by the classification of thick ideals.
\end{proof}

\begin{Rem}
    In general, $(f_i^*)_{i \in I}$ being jointly conservative does not imply that the family is weakly descendable. Indeed, let $k$ be a field and $I$ an infinite set. By \cite{Stevenson14}, the projection functors
    \[
        (\pi_i^*\colon \Der(\textstyle\prod_{i \in I}k) \to \Der(k))
    \]
    are jointly conservative. However,  the localizing ideal $\Loco{(\pi_i)_*k}$ is a proper subcategory of $\Der(\prod_{i \in I}k)$ since $\prod_{i \in I}k$ is not semi-artinian (\cref{exa:absflat}).
\end{Rem}

\begin{Rem}
    The next result is a minor variation on  \cref{prop:surjectivity} which improves on \cite[Corollary 14.24]{BCHS1} by extending the latter to families of geometric functors and weakening the tt-stratification assumption to h-stratification.
\end{Rem}

\begin{Cor}\label{cor:hstrat_surjectivity}
    Let $(f_i^*\colon \cat T \to \cat S_i)_{i\in I}$ be a geometric family and assume that $\cat T$ is h-stratified. Then the following are equivalent:
        \begin{enumerate}
            \item $(f_i^*)_{i\in I}$ is weakly descendable;
            \item $(f_i^*)_{i \in I}$ is jointly conservative;
            \item $(f_i^*)_{i \in I}$ is jointly nil-conservative;
            \item the induced map on homological spectra 
                \[
                    \varphi^h = \sqcup \varphi_i^h\colon \bigsqcup_{i \in I}\Spc^h(\cat S_i^c) \to \Spc^h(\cat T^c) 
                \] 
                is surjective.
    \end{enumerate}
\end{Cor}

\begin{proof}
    The implications $(a) \Rightarrow (b) \Rightarrow (c) \Rightarrow (d)$ are part of \cref{prop:surjectivity}, so it remains to show $(d) \Rightarrow (a)$. Indeed, if $\varphi^h$ is surjective, then 
        \[
            \Spc^h(\cat T^c) = \bigcup_{i \in I} \im \varphi_i^h = \bigcup_{i \in I} \Supph((f_i)_*\unit_{\cat S_i}),
        \]
    where the second equality uses \cite[Theorem 5.12]{Balmer20_bigsupport}. Since $\cat T$ is h-stratified, this implies that $\cat T = \Loco{ (f_i)_*\unit_{\cat S_i} \mid i \in I}$, so $(f_i^*)_{i\in I}$ is weakly descendable.
\end{proof}

\begin{Rem}\label{rem:weak-descendable-best}
    In light of the above corollary and its analogue \cite[Corollary~14.24]{BCHS1}, we regard weak descendability as the appropriate condition for studying descent properties in rigidly-compactly generated tt-geometry.
\end{Rem}

\begin{Exa}
	Let $\cat T$ be a rigidly-compactly generated tt-category. By definition, the family of localizations  $\{ f_{\cat P}^*\colon \cat T \to \cat T_{\cat P}\}_{\cat P \in \Spc(\cat T^c)}$ is weakly descendable if and only if
	\[
        \unit \in \Loco{ f_{\gen(\cat P)^c} \mid \cat P \in \Spc(\cat T^c)}.
    \]
	The following example demonstrates that this need not always be the case:
\end{Exa}

\begin{Exa}\label{ex:absolutelyflat}
	Let $\cat T=\Der(R)$ be the derived category of an absolutely flat ring $R$. Then the local categories are precisely the derived categories of the residue fields  and $f_{\gen(\cat P)^c} \simeq \kappa(\mathfrak p)$ for $\cat P$ corresponding to $\mathfrak p$.  Thus the family of localizations is weakly descendable if and only if 
	\[
        R \in \Loco{ \kappa(\mathfrak p) \mid \mathfrak p \in \Spec(R)}.
    \]
	That is, if and only if $\Der(R)$ has the homological local-to-global principle (recall \cref{exa:DR-h-codetection}). This is the case if and only if $R$ is semi-artinian; see \cref{exa:absflat}. 
\end{Exa}

\begin{Prop}\label{prop:h-LGP-weakly-descendable}
	Let $\cat T$ be a rigidly-compactly generated tt-category. If the homological local-to-global principle holds then the family of localizations
	\[
        \{f_{\cat P}^* \colon \cat T \to \cat T_{\cat P}\}_{\cat P \in \Spc(\cat T^c)}
    \]
	is weakly descendable.
\end{Prop}

\begin{proof}
    By hypothesis and \cref{rem:hLGPannoying}, $\unit \in \Loco{E_{\cat B} \mid \cat B \in \Spc^h(\cat T^c)}$. Hence it suffices to prove that $E_{\cat B} \in \Loco{ (f_{\cat P})_*(\unit)}$ for $\cat P \coloneqq \pi(\cat B)$. This follows from the fact that $E_{\cat B}\otimes(f_{\cat P})_*(\unit)\cong E_{\cat B}$; see \cite[Lemma~8.6]{Zou23bpp}.
\end{proof}

\begin{Rem}
    The converse of \cref{prop:h-LGP-weakly-descendable} does not hold. For example, the homological local-to-global principle does not hold for the derived category of the truncated polynomial ring in \cref{exa:truncatedpolyring} (see \cref{rem:h-LGP-fail}), but the family of localizations is  weakly descendable for the trivial reason that the category is local.
\end{Rem}

\begin{Prop}\label{prop:LGP-weakly-descendable}
	Let $\cat T$ be a rigidly-compactly generated tt-category whose spectrum is weakly noetherian. If the local-to-global principle holds then the family of localizations
	\[
        \{f_{\cat P}^* \colon \cat T \to \cat T_{\cat P}\}_{\cat P \in \Spc(\cat T^c)}
    \]
	is weakly descendable.
\end{Prop}

\begin{proof}
	Recall that $g_{\cat P} = e_{Y} \otimes f_{\gen(\cat P)^c}$ for some Thomason subset $Y$  and $f_{\gen(\cat P)^c} = (f_{\cat P})_*(\unit)$. Thus
	\[
        \Loco{g_{\cat P} \mid \cat P \in \Spc(\cat T^c)} \subseteq \Loco{ (f_{\cat P})_*(\unit) \mid \cat P \in \Spc(\cat T^c)}.
    \]
	The local-to-global principle states that $\unit$ is contained in the left-hand side, while the family is weakly descendable if $\unit$ is contained in the right-hand side.
\end{proof}

\begin{Prop}\label{prop:ttfields2}
    Suppose $\cat T$ admits enough tt-fields, say $\{f_{\cat B}^*\colon \cat T \to \cat F_{\cat B}\}_{\cat B}$. The following are equivalent:
    \begin{enumerate}
    	\item The family $\{f_{\cat B}^*\colon \cat T \to \cat F_{\cat B}\}$ is weakly descendable.
    	\item The homological local-to-global principle holds.
    \end{enumerate}
\end{Prop}

\begin{proof}
    If $\cat T$ has the h-detection property then $\Loco{E_{\cat B}}=\Loco{(f_{\cat B})_*\unit}$ for every~$\cat B$ by \cref{prop:forced-minimal}, which implies the desired equivalence. Certainly h-LGP implies h-detection. To finish the proof, we claim that $(a)$ also implies h-detection. Indeed, since tt-fields are h-stratified (\cref{exa:ttfieldhstrat}), we have $(\varphi^h_{\cat B})^{-1}(\Supph(t)) = \Supph(f_{\cat B}^*(t))$ for every $\cat B$ and every object $t \in \cat T$ by \cref{prop:AV-S-h-strat}. Therefore, the h-detection property is equivalent to the family $\{f_{\cat B}^*\colon \cat T \to \cat F_{\cat B}\}_{\cat B}$ being jointly conservative. The claim thus follows from \cref{prop:surjectivity}.
\end{proof}

\begin{Exa}
    The derived category $\Der(R)$ of a commutative ring $R$ satisfies the h-LGP if and only if the family $\{\Der(R) \to \Der(\kappa(\mathfrak p))\}_{\mathfrak p\in\Spec(R)}$ of residue fields is weakly descendable. 
\end{Exa}

\begin{Rem}
    The following theorem shows that h-stratification \emph{always} descends along a weakly descendable family of geometric functors. Although the proof is not difficult given our preparation, it is one of the key insights of the paper and will have various consequences below.
\end{Rem}

\begin{Thm}\label{thm:h-descent}
	Let $(f_i^*\colon \cat T \to \cat S_i)_{i \in I}$ be a weakly descendable geometric family. Assume that $\cat S_i$ is h-stratified for each $i \in I$. Then $\cat T$ is h-stratified.
\end{Thm}

\begin{proof}
    We will verify condition $(b)$ of \cref{thm:hstratfundamental}. To this end, let $t \in \cat T$ and first fix some $i \in I$. Using \cref{prop:AV-S-h-strat} twice, we obtain equalities
    \[
        \Supph(f_i^*(t)) = (\varphi_i^h)^{-1}(\Supph(t)) = \Supph(\SET{f_i^*(E_{\cat B})}{\cat B \in \Supph(t)}).
    \]
    \Cref{thm:hstratfundamental}$(b)$ applied to $\cat S_i$ then shows that
    \[
        \Loco{f_i^*(t)} = \Loco{f_i^*(E_{\cat B}) \mid \cat B \in \Supph(t)}.
    \]
   Keeping in mind that \eqref{eq:weakdescentcondition} holds by assumption, \cref{lem:weakdescentcharacterization} implies that $\Loco{t} = \Loco{E_{\cat B}\mid\cat B \in \Supph(t)}$, as desired.
\end{proof}

\begin{Par}
    We also include an alternative proof of \cref{thm:h-descent} based on cosupport.
\end{Par}

\begin{proof}[Alternative proof of \cref{thm:h-descent}]
    We will verify condition (c) of \cref{thm:hstratcosupp}. Note that the collection of maps $(\varphi_i^h)$ is jointly surjective by \cref{prop:surjectivity}. Since $(f_i^*)$ is weakly descendable and $\cat S_i$ is h-stratified for all $i$, for any objects $t_1,t_2 \in \cat T$ we have equivalences:
        \begin{align*}
            0 = \ihom{t_1,t_2} & \iff \forall i \in I\colon 0 = f_i^!\ihom{t_1,t_2} \simeq \ihom{f_i^*t_1,f_i^!t_2} & & (\ref{lem:weakdescentcharacterization}) \\
            & \iff \forall i \in I\colon \emptyset = \Supph(f_i^*t_1) \cap \Cosupph(f_i^!t_2) & & (\ref{thm:hstratcosupp}) \\
            & \iff \emptyset = \bigcup_{i \in I} (\varphi_i^h)^{-1}(\Supph(t_1) \cap \Cosupph(t_2)) && (\ref{prop:AV-S-h-strat},\ref{prop:AV-coS-h-strat}) \\
            & \iff \emptyset = \Supph(t_1) \cap \Cosupph(t_2) & & (\ref{prop:surjectivity}),
        \end{align*}
    as desired. 
\end{proof}

\section{Comparison between h-stratification and tt-stratification}\label{sec:comp-strat}

A theory of stratification for tt-categories is developed in \cite{bhs1} under the assumption that the Balmer spectrum is weakly noetherian. Our next goal is to relate this notion of stratification (based on the Balmer spectrum and the Balmer--Favi support) with the homological stratification developed in this paper (based on the homological spectrum and the homological support).

\begin{Def}\label{def:nerves-of-steel}
	We will say that a rigidly-compactly generated tt-category $\cat T$ satisfies the \emph{Nerves of Steel conjecture} (\cite[Remark 5.15]{Balmer20_nilpotence}) if the surjective comparison map
	\[
        \pi\colon \Spc^h(\cat T^c) \to \Spc(\cat T^c)
    \]
	is a bijection.
\end{Def}

\begin{Rem}
    The question of whether every rigidly-compactly generated tt-category satisfies the Nerves of Steel conjecture is an open problem, but it is known to be true in many examples of interest; see \cite[Section 5]{Balmer20_nilpotence}, for example. For the state of the art, see \cite{Hyslop2024}.
\end{Rem}

\begin{Prop}\label{prop:nervesofsteel_Zariskilocal}
    Suppose $\Spc(\cat T^c)=\bigcup_{i\in I}V_i$ is a cover by complements  $V_i$ of Thomason subsets. Then the Nerves of Steel conjecture holds for $\cat T$ if and only if it holds for each $\cat T(V_i)$.
\end{Prop}

\begin{proof}
    The only if part follows immediately from the diagram in \cref{cor:cartesian}. Now suppose each $\cat T(V_i)$ satisfies the Nerves of Steel conjecture. For any $\cat P \in \Spc(\cat T^c)$ there exists some $V_i$ containing $\cat P$. Applying \cref{prop:image-of-finite} to the finite localization $\cat T \to \cat T(V_i)$ we obtain the commutative diagram
	\[\begin{tikzcd}
		\Spc^h(\cat T(V_i)^c) \ar[d,"\simeq"] \ar[hook,r] &\Spc^h(\cat T^c) \ar[d,"\pi"] \\
		\Spc(\cat T(V_i)^c) \ar[hook,r] &\Spc(\cat T^c)
	\end{tikzcd}\]
    which implies that $\pi^{-1}(\{\cat P\})$ consists of exactly one point.
\end{proof}

\begin{Rem}
    In particular, $\cat T$ satisfies the Nerves of Steel conjecture if and only if each local category $\cat T_{\cat P}$ does.
\end{Rem}

\begin{Def}\label{def:tt-stratified}
	We will say that a rigidly-compactly generated tt-category whose spectrum $\Spc(\cat T^c)$ is weakly noetherian is \emph{tt-stratified} if it is stratified in the sense of \cite{bhs1}.
\end{Def}

\begin{Thm}\label{thm:tt=h+NS}
	Let $\cat T$ be a rigidly-compactly generated tt-category with  $\Spc(\cat T^c)$ weakly noetherian. The following are equivalent:
	\begin{enumerate}
		\item $\cat T$ is tt-stratified;
		\item $\cat T$ is h-stratified and the Nerves of Steel conjecture holds for $\cat T$.
	\end{enumerate}
\end{Thm}

\begin{proof}
	$(a)\Rightarrow (b)$: By \cite[Theorem 4.7]{bhs2} if $\cat T$ is tt-stratified, then the Nerves of Steel conjecture holds for $\cat T$, so it remains to show that $\cat T$ is h-stratified. It was proved in \cite[Lemma 3.7]{bhs2} that $\Supp(E_{\cat B})=\{ \pi(\cat B)\}$ for any $\cat B \in \Spc^h(\cat T^c)$. Hence $\Loco{E_{\cat B}} = \Loco{g_{\pi(\cat B)}}$ when $\cat T$ is tt-stratified. Therefore,
	\[
        \Loco{t} = \Loco{g_{\cat P} \mid \cat P \in \Supp(t)} = \Loco{E_{\cat B} \mid \cat B \in \pi^{-1}(\Supp(t))}
    \]
	for any $t \in \cat T$. This coincides with $\Loco{E_{\cat B} \mid \cat B \in \Supph(t)}$ since \cite[Theorem 4.7]{bhs2} establishes that $\pi^{-1}(\Supp(t))=\Supph(t)$ when $\cat T$ is tt-stratified. This establishes that $\cat T$ is h-stratified by \cref{thm:hstratfundamental}.

	$(b)\Rightarrow (a)$: If $\cat T$ is h-stratified then the h-detection property holds. Thus, \cref{prop:support-relations} implies that $\pi(\Supph(t))=\Supp(t))$ for each $t \in \cat T$. It follows that the diagram
    \begin{equation}\label{eq:supportcomparison}
    \begin{tikzcd}
    		\big\{ \text{localizing $\otimes$-ideals of $\cat T$} \big\} \ar[r,"\Supph"] \ar[rd,"\Supp"',end anchor={west}] & \big\{ \text{subsets of $\Spc^h(\cat T^c)$}\big\} \ar[d,"\pi"] \\
    & \big\{ \text{subsets of $\Spc(\cat T^c)$}\big\}
    \end{tikzcd}
    \end{equation}
    commutes. By hypothesis, the horizontal and vertical maps are injective (and thus bijective), hence so is the diagonal map; that is, $\cat T$ is tt-stratified.
\end{proof}

\begin{Cor}\label{cor:tt-stratified-iff-h-strified}
	If the Nerves of Steel conjecture holds for $\cat T$ and $\Spc(\cat T^c)$ is weakly noetherian then $\cat T$ is tt-stratified if and only if it is h-stratified.
\end{Cor}

\begin{proof}
	This immediately follows from \cref{thm:tt=h+NS}.
\end{proof}

\begin{Cor}\label{cor:BHttstratificationdescent}
    Let $(f_i^*\colon \cat T \to \cat S_i)_{i \in I}$ be a weakly descendable geometric family. Assume that $\cat S_i$ is tt-stratified for each $i \in I$ and that $\Spc(\cat T^c)$ is weakly noetherian. If the Nerves of Steel conjecture holds for $\cat T$ then $\cat T$ is tt-stratified.
\end{Cor}

\begin{proof}
    For each $i \in I$, the tt-category $\cat S_i$ is tt-stratified by assumption, hence h-stratified by \cref{thm:tt=h+NS}. Descent for h-stratification as established in \cref{thm:h-descent} then implies that $\cat T$ is h-stratified. Therefore, \cref{thm:tt=h+NS} implies that $\cat T$ is tt-stratified, due to our assumption that $\cat T$ satisfies Nerves of Steel.
\end{proof}

\begin{Par}
    We conclude this section with a proof of \cref{thmx:a} from the introduction.
\end{Par}

\begin{proof}[Proof of \cref{thmx:a}]
    The equivalence between $(a)$ and $(b)$ is \cref{cor:tt-stratified-iff-h-strified}. \Cref{cor:hstrat_surjectivity} gives $(b) \Rightarrow (c)$, while the reverse implication is \cref{thm:h-descent}. 
\end{proof}

\section{Descent for the Nerves of Steel conjecture}\label{sec:descendingNSC}

In order to apply \cref{cor:BHttstratificationdescent} to descend tt-stratification, we desire methods for establishing the Nerves of Steel conjecture. This is the topic of the present section. 

\begin{Lem}\label{lem:corestrict-strongly-closed}
    Let $f^*\colon \cat T \to \cat S$ be a geometric functor and let $V \subseteq \Spc(\cat T^c)$ be the complement of a Thomason subset. If $f^*$ is strongly closed, then the canonical corestriction functor $\cat T(V) \to \cat S(\varphi^{-1}(V))$ is also strongly closed.
\end{Lem}

\begin{proof}
    Let $Y \coloneqq V^c$ and consider the commutative square
    \[
    \begin{tikzcd}
        \cat T \ar[r,"f^*"] \ar[d,"h^*"] & \cat S \ar[d,"k^*"] \\
        \cat T(V) \ar[r,"g^*"] & \cat S(\varphi^{-1}(V))
    \end{tikzcd}
    \]
    where the vertical arrows are the finite localizations and the bottom arrow is the induced  functor. This diagram satisfies the Beck--Chevalley condition $h^* f_* \simeq g_* k^*$; cf.~\cite[Lemma 5.5]{Sanders21pp}. Indeed, this may be checked after applying the conservative $h_*$. For any $s \in \cat S$, we have $h_*h^*f_*(s) \simeq {f_Y \otimes f_*(s)} \simeq {f_*(f^*(f_Y)\otimes s)}\simeq f_*(f_{\varphi^{-1}(Y)} \otimes s) \simeq f_*(k_*k^*(\unit) \otimes s) \simeq f_*k_*k^*(s)\simeq h_*g_*k^*(s)$ where we have invoked \cite[Proposition 5.11]{BalmerSanders17}. With this in hand, consider a compact object $x \in \cat S(\varphi^{-1}(V))^c$. Then $x \oplus \Sigma x = k^*(c)$ for some $c \in \cat S^c$ and it suffices to show that $g_* k^*(c)$ is compact. This follows from our hypothesis, since $g_* k^*(c) \simeq h^*f_*(c)$.
\end{proof}

\begin{Prop}\label{prop:NSCdescent}
    Suppose $(f_i^*\colon\cat T \to \cat S_i)_{i \in I}$ is a jointly nil-conservative family of geometric functors with $\cat S_i$ satisfying the Nerves of Steel conjecture for all $i\in I$. Assume additionally one of the following:
        \begin{enumerate}
            \item the induced map $\varphi\colon \bigsqcup_{i \in I}\Spc(\cat S_i^c) \to \Spc(\cat T^c)$ is injective; or
            \item $(f_i^*)_{i \in I}$ is strongly closed and $\varphi_i$ has discrete fibers for all $i \in I$; or
            \item $(f_i^*)_{i \in I}$ is strongly closed and $\Spc(\cat S_i^c)$ is noetherian for all $i \in I$.
        \end{enumerate}
    Then $\cat T$ satisfies the Nerves of Steel conjecture.
\end{Prop}

\begin{proof}
    Consider the commutative diagram
    \[ 
    \begin{tikzcd}
		\bigsqcup_{i \in I}\Spc^h(\cat S_i^c) \ar[d,"\sqcup \pi_{\cat S_i}","\simeq"'] \ar[r,"\varphi^h"] & \Spc^h(\cat      T^c) \ar[d,"\pi_{\cat T}",two heads] \\
		\bigsqcup_{i \in I}\Spc(\cat S_i^c) \ar[r,"\varphi"] & \Spc(\cat T^c).
	\end{tikzcd}
    \]
    Recall from \cref{prop:surjectivity} that $(f_i^*)$ being jointly nil-conservative implies that~$\varphi$ is also surjective. If $\varphi$ is in addition injective as assumed in $(a)$, then it is bijective. Since $\sqcup\pi_{\cat S_i}$ is bijective, $\varphi^h$ is injective and hence bijective. Therefore $\pi_{\cat T}$ is bijective.

    Now assume either condition $(b)$ or $(c)$. By \cref{cor:cartesian}, the Nerves of Steel conjecture holds also for any finite localization of $\cat S_i$. Moreover, the property of being strongly closed is local in the target by \cref{lem:corestrict-strongly-closed}, while the property of being nil-conservative is local in the target by \cref{prop:surjectivity}[$(c) \Leftrightarrow (d)$] and \cref{cor:cartesian2}. Therefore, we can reduce to the case when $\cat T$ is local and check that the fiber of $\pi_{\cat T}\colon\Spc^h(\cat T^c)\to \Spc(\cat T^c)$ over the unique closed point $\mathfrak m \in \Spc(\cat T^c)$ is a singleton. Since $\varphi$ is surjective, we may choose $i \in I$ and a closed point $\cat M \in \Spc(\cat S_i^c)$ lying over $\mathfrak m$. For the remainder of this proof, we set $\cat S = \cat S_i$ and simplify the notation accordingly.
    
    Observe that $\{\cat M\}$ is Thomason closed if $\Spc(\cat S^c)$ is noetherian or $\varphi$ has discrete fibers. In either case, by \cite[Proposition 2.14]{Balmer05a} there exists some compact object $z \in \cat S^c$ with $\supp(z)=\{\cat M\}$. Replacing $z$ by $z \otimes z^{\vee}$ we may assume that~$z$ is a (weak) ring. Let $\cat A\in\Spc^h(\cat S^c)$ be the unique homological prime such that $\pi_{\cat S}(\cat A) = \cat M$, so that $\Supph(z) = \{\cat A\}$.

    Suppose for a contradiction that there exists a point $\cat B \in \Spc^h(\cat T^c)$ lying over $\mathfrak m$ such that $\cat B\neq\varphi^h(\cat A)$. Recall from \cref{exa:AV-scott-EB} that 
        \[
            \Supph(f^*(E_{\cat B})) = (\varphi^h)^{-1}(\{\cat B\}).
        \]
    It follows that 
        \[
            \Supph(z\otimes f^*(E_{\cat B})) = \Supph(z) \cap \Supph(f^*(E_{\cat B})) = \{\cat A\} \cap (\varphi^h)^{-1}(\{\cat B\}) = \emptyset
        \]
    and hence $f^*(E_{\cat B})\otimes z=0$ because homological support detects weak rings (\cref{rem:hsuppdetectsweakrings}). Therefore, $E_{\cat B} \otimes f_*(z)=0$ by the projection formula. That is $\cat B \not\in \Supph(f_*(z))$. By the hypothesis that $f$ is strongly closed, $f_*(z)$ is compact and hence $\Supph(f_*(z)) = \pi_{\cat T}^{-1}(\supp(f_*(z)))$ by \cref{rem:comparisonforcompacts}. Therefore we have $\cat B \not\in \pi_{\cat T}^{-1}(\supp(f_*(z)))$ so $\mathfrak m = \pi_{\cat T}(\cat B) \not\in \supp(f_*(z))=\varphi(\supp(z))$ where the last equality is by \cite[Theorem~13.13]{BCHS1} since $z$ is a compact weak ring. But $\varphi(\supp(z))$ contains $\varphi(\{\cat M\}) = \mathfrak m$ yielding the desired contradiction.
\end{proof}

\begin{Rem}
    The results of \cref{prop:NSCdescent} are not entirely satisfying; variations on the theme are possible and we expect a more complete statement to exist. Nevertheless, using \cref{prop:NSCdescent}$(a)$, we can provide a variant of the argument used in \cite{Balmer20_nilpotence} to deduce the Nerves of Steel conjecture from nilpotence-type theorems in several examples of interest:    
\end{Rem}

\begin{Prop}\label{prop:nil-cons-NSC}
    Let $\cat T$ be a rigidly-compactly generated tt-category.  Suppose there exists a tt-field  $f_{\cat P}^*\colon \cat T \to \cat F_{\cat P}$ detecting $\cat P$ for each prime $\cat P \in \Spc(\cat T^c)$, i.e., $\cat P$ is in the image of $\Spc(f_{\cat P}^*)$. The following are equivalent:
    \begin{enumerate}
        \item $\cat T$ satisfies the Nerves of Steel conjecture.
        \item The family $(f_{\cat P}^*)_{\cat P \in \Spc(\cat T^c)}$ jointly detects $\otimes$-nilpotence of morphisms with dualizable source.
        \item The family $(f_{\cat P}^*)_{\cat P \in \Spc(\cat T^c)}$ is jointly nil-conservative.
    \end{enumerate}
\end{Prop}

\begin{proof}
    Recall from \cite[Lemma~2.2]{BalmerCameron2021} that every tt-field induces a homological residue field.
    More precisely, for each $\cat P \in \Spc(\cat T^c)$, there exists a homological prime $\cat B \in \Spc^h(\cat T^c)$ lying over $\cat P$ and a commutative diagram
    \[\begin{tikzcd}
        \cat T \ar[rr,bend left,"\overline{\mathbb{h}}_{\cat B}"]\ar[d,"f_{\cat P}^*"'] \ar[r,"\mathbb{h}"] & \text{Mod-}\cat T^c=\cat A \ar[d] \ar[r]
        & \overbar{\cat A}_{\cat B} \ar[dl,"\overbar{F}",end anchor=east]\\
        \cat F_{\cat P}  \ar[r,"\mathbb{h}"] & \text{Mod-}\cat F_{\cat P}^c 
    \end{tikzcd}\]
    where the functor $\overbar{F}$ is faithful; see \cite[(2.3)]{BalmerCameron2021}. It follows from this diagram and the fact that the tt-field $\cat F_{\cat P}$ is phantomless that a morphism in $\cat T$ is killed by $f_{\cat P}^*$ if and only if it is killed by the homological residue field $\overline{\mathbb{h}}_{\cat B}$. Thus, the family $(f_{\cat P}^*)_{\cat P \in \cat \Spc(\cat T^c)}$ jointly detects $\otimes$-nilpotence if and only if the corresponding family of homological residue fields (also indexed on $\Spc(\cat T^c)$) jointly detects $\otimes$-nilpotence. The latter is equivalent to the Nerves of Steel conjecture by \cite[Corollary~4.7 and Theorem~5.4]{Balmer20_nilpotence}. This establishes the equivalence of $(a)$ and $(b)$. On the other hand, contemplating the commutative diagram
    \[ 
        \begin{tikzcd}
		\bigsqcup_{\cat P \in \Spc(\cat T^c)}\Spc^h(\cat F_{\cat P}^c) \ar[d,"\sqcup \pi_{\cat F_{\cat P}}","\simeq"'] \ar[r,"\varphi^h"] & \Spc^h(\cat      T^c) \ar[d,"\pi_{\cat T}",two heads] \\
		\bigsqcup_{\cat P \in \Spc(\cat T^c)}\Spc(\cat F_{\cat P}^c) \ar[r,"\varphi","\simeq"'] & \Spc(\cat T^c),
	\end{tikzcd}
    \]
    we see that $(a)$ implies that the map $\varphi^h\colon\bigsqcup_{\cat P \in \Spc(\cat T^c)} \Spc^h(\cat F_{\cat P}^c)\to \Spc^h(\cat T^c)$ is surjective. This is equivalent to the family being nil-conservative by \cite[Theorem 1.9]{BarthelCastellanaHeardSanders22b}. Thus $(a)\Rightarrow (c)$. Finally, the implication $(c)\Rightarrow (a)$ follows from \cref{prop:NSCdescent}(a).
\end{proof}

\begin{Rem}
    The next result establishes an equivalent formulation for the Nerves of Steel conjecture entirely in terms of the tensor triangular support. It strengthens previous work \cite[Proposition 3.13]{bhs2} and could also be used to give an alternative deduction of \cref{thm:tt=h+NS}.
\end{Rem}

\begin{Prop}\label{prop:tensorformulaweakrings}
    Suppose that $\Spc(\cat T^c)$ is weakly noetherian. The Nerves of Steel conjecture holds for $\cat T$ if and only if $\cat T$ satisfies the tensor product formula for all weak rings, i.e., 
        \[
            \Supp(w_1 \otimes w_2) = \Supp(w_1) \cap \Supp(w_2)
        \]
    for any weak rings $w_1,w_2 \in \cat T$.
\end{Prop}

\begin{proof}
    Assume first that the Nerves of Steel conjecture holds for $\cat T$.  Using \cref{prop:supphweakrings} twice as well as the tensor product formula for homological support \eqref{eq:tensor-product}, we compute
    \begin{align*}
        \Supp(w_1 \otimes w_2) & = \pi \Supph(w_1 \otimes w_2) \\
        & = \pi(\Supph(w_1) \cap \Supph(w_2)) \\
        & = \pi\Supph(w_1) \cap \pi\Supph(w_2) \\
        & = \Supp(w_1) \cap \Supp(w_2),
    \end{align*}
    where the third equality uses the hypothesis that $\pi$ is a bijection.

    Conversely, suppose that the tensor product formula holds for all weak rings.  The argument in the proof of  \cite[Proposition 3.13]{bhs2} then implies the Nerves of Steel conjecture for $\cat T$. Note that the result there assumed the full tensor product formula, but it is only used for the weak rings $E_{\cat B}$.
\end{proof}

\section{Applications to tt-stratification}\label{sec:tt-applications}

We now state our main applications. We start with three descent results which extend those of  \cite[Section 17]{BarthelCastellanaHeardSanders22b} and \cite[Section 12]{BCHNPS_descent} to a family of functors.

\begin{Prop}\label{prop:ttstratdescent}
	Let $(f_i^*\colon \cat T \to \cat S_i)_{i \in I}$ be a weakly descendable geometric family with each $\cat S_i$ tt-stratified. Suppose that $\Spc(\cat T^c)$ is weakly noetherian. Assume additionally one of the following: 
    \begin{enumerate}
        \item the induced map $\varphi\colon \bigsqcup_{i \in I}\Spc(\cat S_i^c) \to \Spc(\cat T^c)$ is injective; or
        \item $(f_i^*)_{i \in I}$ is strongly closed and $\varphi_i$ has discrete fibers for all $i \in I$; or
        \item $(f_i^*)_{i \in I}$ is strongly closed and $\Spc(\cat S_i^c)$ is noetherian for all $i \in I$.
    \end{enumerate}
    Then $\cat T$ is tt-stratified.
\end{Prop}

\begin{proof}
    By \Cref{thm:tt=h+NS} each $\cat S_i$ is h-stratified and satisfies the Nerves of Steel conjecture. The result then follows from \cref{prop:NSCdescent} and \cref{cor:BHttstratificationdescent} since $(f_i^*)_{i \in I}$ is jointly nil-conservative by \cref{prop:surjectivity}.
\end{proof}

\begin{Rem}\label{rem:ttdescent-finite}
    In the situation of \cref{prop:ttstratdescent}$(c)$, if the indexing set $I$ is finite then $\Spc(\cat T^c)$  is automatically noetherian and hence the hypothesis that $\Spc(\cat T^c)$ is weakly noetherian is superfluous.
\end{Rem}

\begin{Cor}\label{cor:ttstratfinitemonodescent}
    Let $\cat C$ be a rigidly-compactly generated symmetric monoidal stable $\infty$-category and let $A \in \CAlg(\cat C^c)$ be a descendable dualizable commutative algebra in $\cat C$. If $\Mod_{\cat C}(A)$  is tt-stratified and has a noetherian spectrum, then $\cat C$ is also tt-stratified and also has a noetherian spectrum.
\end{Cor}

\begin{proof}
Since $A$ is descendable,
the base-change functor $A\otimes-\colon \cat C \to \Mod_{\cat C}(A)$ is conservative and hence induces a surjection $\Spc(\Mod_{\cat C}(A)^c) \to \Spc(\cat C^c)$; see \cite[Theorem~1.4]{BarthelCastellanaHeardSanders22b} or \cite[Proposition~7.8]{BCHNPS_descent}.
Hence $\Spc(\cat C^c)$ is also noetherian. We may then invoke \cref{prop:ttstratdescent}$(c)$.
\end{proof}

\begin{Rem}
    The previous corollary strengthens the results of \cite[Section~12]{BCHNPS_descent} by removing the hypothesis that the descendable algebra $A$ be separable of finite tt-degree.
\end{Rem}

\begin{Par}
	We now turn to a cohomological criterion for stratification.
\end{Par}

\begin{Def}
	A rigidly-compactly generated tt-category with weakly noetherian spectrum is said to be \emph{cohomologically stratified} if it is tt-stratified and the comparison map (\cite{Balmer10b})
    \[
    	\rho\colon \Spc(\cat T^c) \to \Spec^h(\End_{\cat T}^{\bullet}(\unit))
    \]
    is a bijection.
\end{Def}

\begin{Prop}\label{prop:descend-coho-strat}
	Let $(f_i^*\colon \cat T \to \cat S_i)_{i \in I}$ be a weakly descendable geometric family with $\Spc(\cat T^c)$ weakly noetherian. Assume that each $\cat S_i$ is cohomologically stratified and that 
    \[
        \bigsqcup_{i \in I} \Spec^h\End^\bullet_{\cat S_i}(\unit) \xra{\sim} \Spec^h(\End^\bullet_{\cat T}(\unit))
    \]
    is a bijection. Then $\cat T$ is cohomologically stratified.
\end{Prop}

\begin{proof}
	The various comparison maps fit into a commutative diagram
    \[
        \begin{tikzcd}
		\bigsqcup_{i \in I}\Spc^h(\cat S_i^c) \ar[d,"\sqcup\pi_{\cat S_i}"',two heads] \ar[r,two heads] &\Spc^h(\cat T^c) \ar[d,two heads,"\pi_{\cat T}"] \\
		\bigsqcup_{i \in I}\Spc(\cat S_i^c) \ar[r,two heads] \ar[d,"\sqcup\rho_{\cat S_i}"',two heads] &\Spc(\cat T^c) \ar[d,"\rho_{\cat T}"] \\
            \bigsqcup_{i \in I}\Spec^h(\End^\bullet_{\cat S_i}(\unit)) \ar[r,"\simeq"] & \Spec^h(\End^\bullet_{\cat T}(\unit))
	\end{tikzcd}
    \]
    in which the three indicated vertical maps are surjections. The top and middle horizontal maps are also surjections by \cref{prop:surjectivity}. The bottom map is a bijection by assumption, hence $\rho_{\cat T}$ is surjective as well. Since the tt-categories $\cat S_i$ are cohomologically stratified, the left vertical composite is a bijection, which then implies that all maps in this diagram are bijections. The result thus follows from \cref{cor:BHttstratificationdescent}. 
\end{proof}

\section{Equivariant applications}\label{sec:applications}

Throughout this section, $G$ will  denote a compact Lie group.

\begin{Rem}
    For foundational material on $\infty$-categories of equivariant $G$-spectra, we refer the reader to \cite{MathewNaumannNoel17,bgh_balmer,BCHNP1}, which are based on the more classical references \cite{LewisMaySteinbergerMcClure86,Alaska96, MandellMay02}. We write $\Sp_G$ for this category. Recall that this is a presentably symmetric monoidal stable $\infty$-category which is rigidly-compactly generated by the orbits~$G/H_+$ associated to the (conjugacy classes) of closed subgroups $H \le G$. Here, we omit the suspension spectrum from our notation.
\end{Rem}

\begin{Rem}
    For any closed subgroup $H \le G$, we have a geometric fixed point functor
    \[
        \Phi_G^H \colon \Sp_G \to \Sp.
    \]
    Importantly, these functors are jointly conservative \cite[Proposition 3.3.10]{Schwede18_global}. Moreover, for any commutative algebra\footnote{i.e., $\mathbb{E}_{\infty}$-monoid in $G$-spectra, without any additional norm structures.} $R \in  \CAlg(\Sp_G)$, we write $\Mod_G(R)$ for $\Mod_{\Sp_G}(R)$, and we have induced functors
    \[
        \Phi_G^H \colon \Mod_G(R) \to \Mod(\Phi^HR)
    \]
    which are still jointly conservative. In fact, the geometric fixed points functors only depend on the conjugacy class of $H$ inside of $G$, and so it suffices to consider a single representative for each such conjugacy class. 
\end{Rem}

\begin{Not}
    We let $\Sub(G)$ denote the set of closed subgroups of $G$ and $\Sub(G)/G$ the set of closed subgroups of $G$ up to conjugation. 
\end{Not}

\begin{Rec}\label{rec:equivariant-facts}
    We will use a number of facts about the equivariant homotopy category in the following result. In particular, recall that for $H \le G$ we have a restriction functor $\mathrm{res}^G_H \colon \Sp_G \to \Sp_H$ that admits left and right adjoints $\ind_H^G$ and $\coind_H^G$, respectively. If $G$ is finite, then these left and right adjoints are naturally isomorphic. This is not true in general, but rather there is a natural equivalence
    \[
        \coind_H^G(Y) \simeq \ind_H^G(Y \otimes S^L),
    \]
    where $L$ is the tangent $H$-representation of $G/H$. Moreover, we have 
    \[
        \ind_H^G(S^0_H) \simeq G/H_+. 
    \]
    Finally, the projection formula holds in this context: there is a natural equivalence
    \[
        \coind_H^G(\mathrm{res}^G_H(X) \otimes Y) \simeq  X\otimes \coind_H^G(Y).
    \]
\end{Rec}

\begin{Lem}\label{lem:eqwdescent}
    Let $G$ be compact Lie group and let $R \in \CAlg(\Sp_G)$ be a commutative equivariant ring spectrum. Then the collection of geometric fixed points
        \[
            \xymatrix{\Phi_G=(\Phi_G^H)\colon \Mod_G(R) \ar[r] & \prod_{H \in \Sub(G)/G}\Mod(\Phi^HR)}
        \]
    is weakly descendable.
\end{Lem}

\begin{proof}
    Let us write $\Psi_G^H$ for the right adjoint of $\Phi_G^H$ (which sometimes goes by the name of geometric inflation). We denote the collection of conjugacy classes of proper closed subgroups of $G$ by $\cP$. We will establish that
        \begin{equation}\label{eq:induction}
            \Mod_G(R) =  \Locs{\Psi_G^K(\unit) \mid K \in \Sub(G)/G}.
        \end{equation}
    Note that we can define an ordering on compact Lie groups by the pair $(\dim H,\pi_0H)$ ordered lexicographically. Specifically, $H \le G$ if and only if $\dim(H) \leq \dim (G)$ or $\dim(H) = \dim(G)$ and the number of connected components of $H$ is less than or equal to the number of connected components of $G$. We will prove \eqref{eq:induction} by induction, with the base case $G = e$ being tautological. We then assume that~\eqref{eq:induction} holds for all  $H \lneq G$, and let $M \in \Mod_G(R)$. 
    
    If $G = H$, then $\Phi^G_G$ is a finite localization away from $\Loco{G/H_+ \otimes R \mid H \in \cP}$, factoring as
        \[
            \Phi_G^G\colon \Mod_G(R) \to \Mod_G(R \otimes \tilde{E}\cP) \simeq \Mod(\Phi_G^GR),
        \]
    where the first map is base-change and the second map is the symmetric monoidal equivalence induced by taking categorical $G$-fixed points. This implies that
        \begin{equation}\label{eq:psi1}
            \Psi_G^G(\unit) \simeq  R \otimes \tilde{E}\cP.
        \end{equation}
    This also follows from \cite[Definition 4.1 and Proposition 4.4]{Hill2012equivariant}. We claim that 
        \begin{equation}\label{eq:psi2}
            M \otimes \tilde{E}\cP \in \Locs{R \otimes \tilde{E}\cP} = \Locs{\Psi_G^G(\unit)}.
        \end{equation} 
    It suffices to check this for the generators of $\Mod_G(R)$, i.e., for $M$ of the form $G/K_+ \otimes R$. The case $K=G$ is \eqref{eq:psi1}, while for $K \lneq G$ the claim holds since $G/K_+ \otimes \tilde{E}\cP = 0$. 
 
    Now let $H$ be a proper closed subgroup of $G$ and consider subgroups $K \le H$. The corresponding geometric fixed point functors then factor as
        \[
            \Phi_G^K\colon \Mod_G(R) \xrightarrow{\res} \Mod_H(R) \xrightarrow{\Phi_H^K}  \Mod(\Phi_H^KR) \simeq \Mod(\Phi_G^KR).
        \]
    By the inductive hypothesis
        \[
            \Mod_H(R) = \Locs{\Psi_H^K\unit\mid K \le H}
        \]
    in $\Mod_H(R)$. In particular, $\res_H(M) \otimes S^L \in \Locs{\Psi_H^K\unit\mid K \le H}$, where $L$ is the tangent $H$-representation of $G/H$. Applying coinduction, we get
        \[
            \coind_H^G(\res_H(M) \otimes S^L) \in \coind_H^G \Locs{\Psi_H^K\unit\mid K \le H} \subseteq  \Locs{\coind_H^G\Psi_H^K\unit \mid K \le H}.
        \]
    By \Cref{rec:equivariant-facts} we have
        \begin{align*}
            \coind_H^G(\res_H(M) \otimes S^L) & \simeq \coind_H^G(S^L) \otimes M \\
            & \simeq \ind_H^G(S_H^0) \otimes M \\
            & \simeq G/H_+ \otimes M.
        \end{align*}
    Since $\coind_H^G\Psi_H^K \simeq \Psi_G^K$, this implies that
        \begin{equation}\label{eq:psi3}
            G/H_+ \otimes M \in \Locs{\Psi_G^K(\unit) \mid K \le H}.
        \end{equation}
    The isotropy separation sequence takes the form
        \[
            M \otimes E\cP_+ \to M \to M \otimes \tilde{E}\cP.
        \]
    We claim that $M \otimes E\cP_+ \in \Locs{G/H_+\otimes M \mid H \in \cP}$. Indeed, as a consequence of work of Illman \cite{Illman83} on the existence of $G$-triangulations (see \cite[Remark~3.6]{Sanders19}), we have 
        \[
            \Thicks{G/H_+ \mid H \in \cP} = \Thickid{G/H_+ \mid H \in \cP}.
        \]
    It follows that 
        \[
            E\cP_+ \in \Loco{G/H_+ \mid H \in \cP} = \Locs{G/H_+ \mid H \in \cP},
        \]
    as desired. Consequently, from isotropy separation and the previous claim we get
        \[
            M \in \Locs{M \otimes \tilde{E}\cP, M \otimes E\cP_+} \subseteq \Locs{M \otimes \tilde{E}\cP , G/H_+\otimes M \mid H \in \cP}.
        \]
    Substituting \eqref{eq:psi2} and \eqref{eq:psi3} thus gives
        \[
            M \in \Locs{\Psi_G^K(\unit) \mid K \le G},
        \]
    which finishes the proof.
\end{proof}

\begin{Lem}\label{lem:eqspc}
    Let $R\in \CAlg(\Sp_G)$ be a commutative equivariant ring spectrum.
    The spectra of $\Mod_G(R)^c$ decompose set-theoretically as 
        \[
            \Spc(\Mod_G(R)^c) = \bigsqcup_{H \in \Sub(G)/G} \im \varphi_H,
            \quad
            \Spc^h(\Mod_G(R)^c) = \bigsqcup_{H \in \Sub(G)/G} \im \varphi^h_H,
        \]
    where $\varphi_H = \Spc(\Phi^H)$ and $\varphi^h_H = \Spc^h(\Phi^H)$ are the maps on spectra induced by $H$-geometric fixed points, respectively.
\end{Lem}

\begin{proof}
    First we show that the claim is true for $R=S_G^0$. Indeed, by \cite[Theorem 3.14]{bgh_balmer}, the analogous result is true for the Balmer spectrum. Because the Nerves of Steel conjecture holds for $\Sp_G$ by \cite[Corollary 5.10]{Balmer20_nilpotence} (which in turn relies on \cite[Corollary 3.18]{bgh_balmer}), the claim follows. 
   
    For the general case, consider the commutative diagram
        \[
        \begin{tikzcd}
            \Spc^h(\Mod(\Phi^HR)^c) \ar[r,"\varphi_H^h"] \ar[d] & \Spc^h(\Mod_G(R)^c) \ar[d] \\
            \Spc^h(\Sp^c) \ar[r] & \Spc^h(\Sp_G^c)
        \end{tikzcd}
        \]
    induced by base-change $S_G^0\to R$ and the $H$-geometric fixed point functors for $\Mod_G(R)$ and $\Sp_G$. This diagram shows that the images of the maps $\varphi^h_H$ are disjoint in $\Spc^h(\Mod_G(R)^c)$ for non-conjugate subgroups, since this is true for the $R=S^0_G$ case just mentioned. As a consequence of \cref{lem:eqwdescent} and \cref{prop:surjectivity}, these images also cover the homological spectrum, which gives the desired decomposition of the homological spectrum. The analogous statement for the tt-spectrum is proved similarly. 
\end{proof}

\begin{Prop}\label{prop:eqhstratdescent}
    Let $G$ be a compact Lie group and $R \in \CAlg(\Sp_G)$ a commutative equivariant ring spectrum. If $\Mod(\Phi^HR)$ is h-stratified for every closed subgroup $H \le G$, then $\Mod_G(R)$ is h-stratified. In this case, h-support and geometric fixed points induce a bijection
        \[ 
            \big\{ \text{localizing ideals of $\Mod_G(R)$} \big\} \xra{\cong} \big\{ \text{subsets of $\bigsqcup_{H \in \Sub(G)/G} \im \varphi^h_H$}\big\}.
        \]
\end{Prop}

\begin{proof}
    The first part follows from \cref{thm:h-descent}, which applies  due to \cref{lem:eqwdescent}. The claimed bijection is then \cref{thm:hstratfundamental} combined with \cref{lem:eqspc}.
\end{proof}

\begin{Rem}
    We next make the previous proposition more explicit in the special case of commutative equivariant ring spectra with trivial action, i.e., those that arise via inflation from (non-equivariant) commutative ring spectra.  This provides a generalization of \cite[Theorem 15.1]{bhs1}, extending it from finite groups to compact Lie groups and removing any topological assumptions on the spectrum.
\end{Rem}

\begin{Thm}\label{thm:infleq_stratification}
    Let $G$ be a compact Lie group and write $R_G = \infl_GR$ for the inflation of a commutative ring spectrum $R\in \CAlg(\Sp)$. If $\Mod(R)$ is h-stratified and satisfies the Nerves of Steel conjecture, then $\pi\circ \Supph$ induces a bijection
        \[
            \begin{Bmatrix}
                \text{localizing ideals} \\
                of \Mod_G(R_G)
            \end{Bmatrix}
                \xrightarrow{\sim}
            \begin{Bmatrix}
                \text{subsets of } \\
                \bigsqcup_{H \in \Sub(G)/G} \Spc(\Mod(R)^c)
            \end{Bmatrix}.
        \]
\end{Thm}

\begin{proof}
    Since the geometric fixed points for $R_G$ are split by inflation, 
        \[
            \varphi_H\colon \Spc(\Mod(R)^c) \to \Spc(\Mod_G(R_G)^c)
        \]
    is a homeomorphism onto its image for any subgroup $H$ in $G$. In particular, using \cref{lem:eqspc}, these maps induce a bijection
        \[
            \Spc(\Mod_G(R_G)^c) \simeq \bigsqcup_{H \in \Sub(G)/G} \Spc(\Mod(R)^c).
        \]
    By \cref{prop:NSCdescent}(a), the Nerves of Steel conjecture descends to $\Mod_G(R_G)$, i.e., $\pi\colon \Spc^h(\Mod_G(R_G)^c) \to \Spc(\Mod_G(R_G)^c)$ is a bijection. \Cref{prop:eqhstratdescent} then implies that $\Mod_G(R_G)$ is h-stratified, as well as the claimed parametrization of localizing ideals.
\end{proof}

\begin{Rem}\label{rem:infleq_supportidentification}
    In fact, an elaboration on the proof of \cref{lem:eqwdescent} can be used to show that $\Spc(\Mod_G(R_G)^c)$ is weakly noetherian if $\Spc(\Mod(R)^c)$ is. Moreover, in this case the function $\pi\circ \Supph$ identifies with the Balmer--Favi notion of support. 
\end{Rem}

\section{Open questions}\label{sec:open-questions}

We end the paper with a few open questions.

\begin{Rem}
    We have defined h-stratification for any rigidly-compactly generated tt-category (\cref{def:hstratification}), but we only defined tt-stratification for those~$\cat T$ whose spectrum $\Spc(\cat T^c)$ is weakly noetherian (\cref{def:tt-stratified}). This is because the latter topological condition is  needed for the construction of the Balmer--Favi support.  However, W.~Sanders \cite{BillySanders2017pp} has introduced a generalization of the Balmer--Favi support which does not require the spectrum to be weakly noetherian. The caveat is that the Balmer--Favi--Sanders support of an object (or localizing ideal) is always closed with respect to the so-called localizing topology on $\Spc(\cat T^c)$; see \cite[Theorem~4.2]{BillySanders2017pp}.\footnote{A spectral space is weakly noetherian if and only if its localizing topology is discrete. See \cite[Section 2]{Zou23bpp} for further information about the localizing topology.} Using this theory of support, one can extend the definition of \mbox{tt-stratification} to arbitrary rigidly-compactly generated tt-categories as in \cite[Definition 8.1]{Zou23bpp}. However, it turns out that if $\cat T$ is tt-stratified in this \emph{a priori} more general sense, then the spectrum $\Spc(\cat T^c)$ is necessarily weakly noetherian; see \cite[Theorem 8.13]{Zou23bpp}. Thus, using the Balmer--Favi--Sanders support does not provide a more general notion of tt-stratification. With this and \cref{thm:tt=h+NS} in mind, it is natural to ask:
\end{Rem}

\begin{Que}\label{que:wN}
    Suppose $\cat T$ is a rigidly-compactly generated tt-category. If $\cat T$ is \mbox{h-stratified} and satisfies the Nerves of Steel conjecture, does it follow that $\Spc(\cat T^c)$ is weakly noetherian?
\end{Que}

\begin{Rem}
    We do not even know the  answer for the derived category of a commutative ring $\Der(R)$. Recall from \cref{exa:DR-h-strat} that in this case the Nerves of Steel conjecture always holds and  h-stratification is equivalent to
    \begin{equation}\label{eq:der-r-condition}
    R \in \Loco{\kappa(\mathfrak p)\mid \mathfrak p \in \Spec(R)} \tag{\textasteriskcentered}.  
    \end{equation}
    Thus, a negative answer to the following question would also provide a negative answer to \Cref{que:wN}. 
\end{Rem}

\begin{Que}
    If \eqref{eq:der-r-condition} holds, does it follow that $\Spec(R)$ is weakly noetherian?
\end{Que}

\begin{Exa}
    For a field $k$, the spectrum of the ring $k[x_1,x_2,\ldots]$ is not weakly noetherian; this follows, for example, from \cite[Proposition 4.13]{BCHS1}. Even for this relatively simple non-noetherian ring, we are unable to determine  whether \eqref{eq:der-r-condition} holds.
\end{Exa}

\begin{Rem}\label{remark:support-relations-non-noetherian}
    As noted above, for an object $t \in \cat T$, the Balmer--Favi--Sanders support $\Supp(t)$ is always closed in the localizing topology on $\Spc(\cat T^c)$. On the other hand, since $\Supph(\coprod_{\cat B \in S} E_{\cat B}) = S$ for any subset $S \subseteq \Spc(\cat T^c)$,  the homological support $\Supph(t)$ takes values in arbitrary subsets of $\Spc^h(\cat T^c)$.  This complicates the relationship between $\pi(\Supph(t))$ and $\Supp(t)$. Nonetheless, one has the following strengthening of \Cref{prop:support-relations}; for unfamiliar terminology, we refer the reader to \cite{Zou23bpp}. 
\end{Rem}
 
\begin{Prop}\label{prop:supph-loc-closure}
    Suppose that h-detection holds. Then the following holds for any $t \in \cat T$:
    \[
    \overline{\pi(\Supph(t))}^{\mathrm{loc}} = \Supp(t),
    \]
    where the left-hand side denotes the closure of $\pi(\Supph t)$ in the localizing topology on $\Spc(\cat T^c)$.
\end{Prop}

\begin{proof}
    We always have $\pi(\Supph(t)) \subseteq \Supp(t)$ by \cite[Lemma 8.5]{Zou23bpp} and the latter is localizing closed, hence $ \overline{\pi(\Supph(t))}^{\mathrm{loc}} \subseteq \Supp(t)$. For the reverse inclusion, recall that weakly visible subsets form a basis of open subsets for the localizing topology. If $\cat P \notin \overline{\pi(\Supph(t))}^{\mathrm{loc}}$ then there exists a weakly visible subset $W \ni \cat P$ such that $\pi(\Supph (t))\cap W = \emptyset$ and hence $\pi^{-1}(W) \cap \Supph (t) = \emptyset$. By the tensor product formula and \cref{lem:supph-gW}, we obtain $\Supph(t \otimes g_W) = \emptyset$. It then follows from h-detection that $t \otimes g_W = 0$. Therefore, $\cat P \not\in \Supp(t)$ 
    by the definition of the Balmer--Favi--Sanders support; see \cite[Definition~5.17]{Zou23bpp}.
\end{proof}

\begin{Par}
We also have the following strengthening of
\cref{prop:supphweakrings}:
\end{Par}

\begin{Prop}\label{prop:supph-loc-closure-wring}
    For any weak ring $w \in \cat T$, we have
    \[
    \overline{\pi(\Supph(w))}^{\mathrm{loc}} = \Supp(w).
    \]
\end{Prop}

\begin{proof}
The $\subseteq$ inclusion always holds, as explained in the proof of \cref{prop:supph-loc-closure}. Moreover, as before, if $\cat P \not\in \overline{\pi(\Supph(w))}^{\mathrm{loc}}$ then there exists a weakly visible subset $W \ni \cat P$ such that $\pi^{-1}(W) \cap \Supph(w) = \emptyset$
and it suffices to establish that ${g_W \otimes w=0}$.
Write $W=Y_1 \cap Y_2^c$. Since $Y_1$ is a union of Thomason closed sets, we may assume without loss of generality that $Y_1$ is a Thomason closed set. Then $Y_1=\supp(x)$ for some $x\in \cat T^c$ and replacing $x$ by $x \otimes x^{\vee}$ we may assume that~$x$ is a (weak) ring. Observe then that 
    \[
        \pi^{-1}(W) = \pi^{-1}(\supp(x)) \cap \pi^{-1}(Y_2^c) = \Supph(x) \cap \Supph(f_{Y_2})
    \]
by \cref{rem:comparisonforcompacts} and \cref{lem:supph-gW}.
Hence 
\[\emptyset = \Supph(x) \cap \Supph(f_{Y_2}) \cap \Supph(w) = \Supph(x \otimes f_{Y_2} \otimes w).\]
This implies that $x \otimes f_{Y_2} \otimes w=0$ since the homological support has the detection property for weak rings (\cref{rem:hsuppdetectsweakrings}).
We conclude that $e_{\supp(x)} \otimes f_{Y_2} \otimes w =0$.
That is, $g_W \otimes x =0$, as desired.
\end{proof}

\begin{Rem}
    In general, we do not have a complete understanding of the relationship between the three notions of support considered in this paper (\cref{rem:support-and-cosupports}). We recalled in \Cref{prop:supph-in-suppn} that the inclusion
    \[
    \Supph(t) \subseteq \Supphnaive(t)
    \]
    always holds, and showed  that it is actually an equality under suitable conditions (for example h-detection). In \cref{cor:hLGP} we showed moreover that, assuming \mbox{h-codetection,} the h-LGP is equivalent to the equality of the two notions of support. Given an unconditional equality between $\Supph$ and $\Supphnaive$, some portion of this paper would simplify; hence, we are led to ask the following:
\end{Rem}

\begin{Que}\label{que:supph=suppn}
    Is there always an equality 
    \[
    \Supph(t) = \Supphnaive(t)?
    \]
\end{Que}

\begin{Exa}
    By \Cref{prop:naive-tt-fields} this holds whenever $\cat T$ has enough tt-fields. For example, this holds in the derived category $\Der(R)$ of a commutative ring. 
\end{Exa}

\begin{Rem}
    Note that we do not know whether $\Supphnaive(t)$ satisfies the tensor-product property, while the homological support always does, so in fact there is an inclusion
    \[
    \Supph(t) \subseteq \bigcap_{k \ge 1} \Supphnaive(t^{\otimes k}).
    \]
    One can therefore ask the following \emph{a priori} weaker variant of \Cref{que:supph=suppn}:
\end{Rem}

\begin{Que}
    Is there always an equality 
     \[
     \Supph(t) = \bigcap_{k \ge 1} \Supphnaive(t^{\otimes k})?
     \]
\end{Que}

\begin{Rem}
    We have introduced a notion of homological cosupport and proved in \Cref{lem:half-hom} that 
	\begin{equation}\label{eq:full-cosupp}
    \Cosupp^h(\ihom{t_1,t_2}) = \Supph(t_1) \cap \Cosupp^h(t_2) 
    \end{equation}
	if $t_1$ is compact or $t_1 = g_W$ for a weakly visible subset $W \subseteq \Spc(\cat T^c)$. Assuming that the h-LGP holds, asking that \eqref{eq:full-cosupp} holds for all $t_1,t_2 \in \cat T$ is equivalent to h-stratification; see \Cref{prop:hstratcosupp}.   On the other hand, we always have 
    \[
    \Supph(t_1 \otimes t_2) = \Supph(t_1) \cap \Supph(t_2)
    \]
    for \emph{any} $t_1,t_2 \in \cat T$. 
\end{Rem}

\begin{Que}
    Is there a variant of homological cosupport which always satisfies
    \[
    \Cosupp_{\mathrm{var}}^h(\ihom{t_1,t_2}) = \Supph(t_1) \cap \Cosupp_{\mathrm{var}}^h(t_2)? 
    \]
\end{Que}

\bibliographystyle{alpha}\bibliography{bibliography}

\end{document}